\documentclass[12pt,leqno]{article}
\usepackage{amssymb,amsfonts,amsmath,amsthm,amscd,mathrsfs}
\usepackage{graphicx}
\def\IE{{\mathbb E}}
\def\IP{{\mathbb P}}
\def\IR{{\mathbb R}}
\def\IN{{\mathbb N}}
\def\IL{{\mathbb L}}

\def\IZ{{\mathbb Z}}

\def\n{\noindent}
\def\dsl{\textstyle\sum\limits}

\def\dis{\displaystyle}
\def\o{\omega}
\def\fr{\mbox{\footnotesize $\dis\frac{1}{2}$}}

\def\ov{\overline}
\def\ve{\varepsilon}
\def\f{\footnotesize}
\def\r{\rightarrow}
\def\point{{\mbox{\large $.$}}}
\def\wh{\widehat}
\def\wt{\widetilde}

\def\cC{{\cal C}}

\def\cE{{\cal E}}
\def\cL{{\cal L}}

\def\cE{{\cal E}}
\def\cI{{\cal I}}

\def\cF{{\cal F}}

\def\cV{{\cal V}}

\dimendef\dimen=0

\newtheorem{theorem}{Theorem}[section]
\newtheorem{lemma}[theorem]{Lemma}
\newtheorem{corollary}[theorem]{Corollary}
\newtheorem{proposition}[theorem]{Proposition}
\newtheorem{remark}[theorem]{Remark}

\begin{document}

\noindent

~

\bigskip
\begin{center}
{\bf DISCONNECTION AND LEVEL-SET PERCOLATION \\ FOR THE GAUSSIAN FREE FIELD}
\end{center}

\begin{center}
Alain-Sol Sznitman
\end{center}

\begin{center}
{\it Dedicated to the memory of Kiyosi Ito on the occasion \\ of the 100th anniversary of his birth}
\end{center}

\bigskip\bigskip
\begin{abstract}
We study the level-set percolation of the Gaussian free field on $\IZ^d$, $d \ge 3$. We consider a level $\alpha$ such that the excursion-set of the Gaussian free field above $\alpha$ percolates. We derive large deviation estimates on the probability that the excursion-set of the Gaussian free field {\it below} the level $\alpha$ disconnects a box of large side-length from the boundary of a larger homothetic box. It remains an open question whether our asymptotic upper and lower bounds are matching. With the help of a recent work of Lupu \cite{Lupu}, we are able to infer some asymptotic upper bounds for similar disconnection problems by random interlacements, or by simple random walk.
\end{abstract}

\vfill 

\n
Departement Mathematik \\
ETH Z\"urich\\
CH-8092 Z\"urich\\
Switzerland

\vfill

\n
$\overline{~~~~~~~~}$
\\
{\footnotesize{This research was supported in part by a Joseph Meyerhoff Visiting Professorship at the Weizmann Institute of Science in Rehovot, Israel.}}
~
\newpage
\thispagestyle{empty}
~

\newpage
\setcounter{page}{1}

 \setcounter{section}{-1}
 
 \section{Introduction}
 \setcounter{equation}{0}
 
Level-set percolation for the Gaussian free field has been of interest for quite some time by now. It is an eminent representative of percolation with long range dependence, see \cite{LeboSale86}, \cite{BricLeboMaes87},  \cite{Gare04}, \cite{Mari07}. Recently, methods stemming from the study of random interlacements have been successfully applied to the model, see \cite{RodrSzni13}, \cite{PopoRath}, and close links between random interlacements and the Gaussian free field have emerged, see \cite{Szni12b}, \cite{Lupu}. Motivated by the study of the shape of a large finite cluster at the origin in supercritical Bernoulli percolation, see \cite{AlexChayChay90}, \cite{Cerf00}, \cite{Grim99}, and by large deviation controls on the occupation-time profile of random interlacements, see \cite{LiSzni15}, asymptotic lower bounds on the probability of disconnection of a large macroscopic body by random interlacements, when the vacant set is in a percolative regime, have been derived in \cite{LiSzni14}. However, so far, there have not been any matching upper bounds to even capture the principal order of exponential decay. In this article, our main object is to investigate similar questions for the level-set percolation of the Gaussian free field on $\IZ^d$, $d \ge 3$. We consider a level $\alpha$ such that the excursion-set of the Gaussian free field above $\alpha$ percolates. We derive upper and lower bounds on the probability that a box of large side-length gets disconnected from the boundary of a larger homothetic box by the excursion-set of the Gaussian free field below level $\alpha$. As a by-product of our results and the recent improvement in \cite{Lupu} of the isomorphism theorem of \cite{Szni12b}, we also derive some upper bounds for the disconnection of a large box by random interlacements, or by simple random walk, which go beyond the current state of knowledge.

\medskip
We will now describe the model and our results more precisely. We consider $\IZ^d$, $d \ge 3$, and denote by $\IP$ the canonical law on $\IR^{\IZ^d}$ of the discrete Gaussian free field, and by $\varphi = (\varphi_x)_{x \in \IZ^d}$, the canonical process, so that under $\IP$,
\begin{equation}\label{0.1}
\begin{array}{l}
\mbox{$\varphi$ is a centered Gaussian process with covariance}
\\
\mbox{$\IE[\varphi_x \varphi_y] = g(x,y)$, for all $x,y \in \IZ^d$},
\end{array}
\end{equation}
where $g(\cdot,\cdot)$ stands for the Green function of the simple random walk on $\IZ^d$, see (\ref{1.1}).

\medskip
Given $\alpha$ in $\IR$, the excursion-set, or level-set, above $\alpha$ is defined as 
\begin{equation}\label{0.2}
E^{\ge \alpha} = \{x \in \IZ^d; \varphi_x \ge \alpha\}.
\end{equation}

\n
Combining results of  \cite{BricLeboMaes87}, \cite{RodrSzni13}, one knows that there is a critical value 
\begin{equation}\label{0.3}
0 \le h_* < \infty\,,
\end{equation}
such that
\begin{equation}\label{0.4}
\begin{array}{l}
\mbox{for $\alpha > h_*$, $\IP$-a.s., $E^{\ge \alpha}$ only has finite connected components},
\\
\mbox{for $\alpha < h_*$, $\IP$-a.s., $E^{\ge \alpha}$ has a unique infinite connected component}.
\end{array}
\end{equation}
One can further introduce a critical value
\begin{equation}\label{0.5}
h_{**} = \inf\{\alpha \in \IR; \liminf\limits_{L} \IP[B_L \stackrel{\ge \alpha}{\longleftrightarrow} \partial B_{2 L}] = 0\},
\end{equation}

\n
where $B_L$ stands for the sup-norm closed ball on $\IZ^d$ with center $0$ and radius $L$, $\partial B_{2L}$ for the boundary of $B_{2L}$ (see the beginning of Section 1), and the event under the probability corresponds to the existence of a nearest neighbor path in the excursion-set $E^{\ge \alpha}$ (see (\ref{0.2})) starting in $B_L$ and ending in $\partial B_{2L}$. One can show, see \cite{PopoRath},  \cite{RodrSzni13}, that $h_{**}$ is finite and that for $\alpha > h_{**}$ the excursion set $E^{\ge \alpha}$ is in a strongly non-percolative regime, see (\ref{1.19}), with a stretched exponential decay of the two-point function $\IP[0 \stackrel{\ge \alpha}{\longleftrightarrow}  x]$ (in fact, an exponential decay, when $d \ge 4$, see \cite{PopoRath}). It is a simple matter to see that $h_* \le h_{**}$, but an important open problem whether the equality $h_* = h_{**}$ actually holds. The positivity of $h_*$ when $d$ is large has been established in  \cite{RodrSzni13}, the (matching) principal asymptotic behaviors $h_* \sim h_{**} \sim \sqrt{2 \log d}$, as $d \r \infty$, have been shown in \cite{DrewRodr}.

\medskip
The main object of this article is to investigate the large $N$ asymptotic behavior of the probability of the disconnection event
\begin{equation}\label{0.6}
A_N =  \big\{\partial B_N \overset{^{\mbox{\footnotesize $\ge \alpha$}}}{\mbox{\Large $\longleftrightarrow$}} \hspace{-4ex}/ \quad \;S_N\big\}
\end{equation}

\n
where there is no nearest neighbor path in $E^{\ge \alpha}$ going from $\partial B_N$ to $S_N = \{x \in \IZ^d$; $|x|_\infty = [MN]\}$, where $|\cdot |_\infty$ stands for the sup-norm, $M > 1$ is a fixed number, and $[MN]$ denotes the integer part of $MN$. The choice of the specific form (\ref{0.6}) of $A_N$ (rather than for instance $\{B_N \stackrel{\ge \,\; \alpha}{\longleftrightarrow}\hspace{-3.8ex}/  \quad \partial B_{MN}\}$) is not essential and mainly motivated by the use of certain contour arguments in Section 3, where the choice (\ref{0.6}) is convenient. 

\medskip
In the strongly non-percolative regime $\alpha > h_{**}$ of the excursion set $E^{\ge \alpha}$, the event $A_N$ becomes typical for large $N$, and (see (\ref{1.20}))   
\begin{equation}\label{0.7}
\lim\limits_N \IP [A_N] = 1, \; \mbox{when $\alpha > h_{**}$}.
\end{equation}

\n
Our main interest in this work lies in the percolative regime, when $A_N$ becomes atypical for large $N$ (for instance, $\lim_N \IP[A_N] = 0$, when $\alpha < h_*$). Let us now describe the asymptotic controls that we obtain.

\medskip
In Theorem \ref{theo2.1} of Section 2 we derive the asymptotic lower bound
\begin{equation}\label{0.8}
\liminf\limits_{N} \; \mbox{\f $\dis\frac{1}{N^{d-2}}$} \; \log \IP[A_N] \ge - \mbox{\f $\dis\frac{1}{2d}$} \;(h_{**} - \alpha)^2 \, {\rm cap}_{\IR^d}([-1,1]^d), \; \mbox{for $\alpha \le h_{**}$}, \end{equation}

\n
where ${\rm cap}_{\IR^d}([-1,1]^d)$ stands for the Brownian capacity of $[-1,1]^d$, see for instance p.~58 of \cite{PortSton78}. This lower bound comes as a rather direct application of the change of probability measure approach, see (\ref{1.22}), and the controls on the strongly non-percolative regime in (\ref{1.19}). It is much simpler to establish than the corresponding lower bound for random interlacements that was derived in \cite{LiSzni14}.

\medskip
As far as upper bounds are concerned, we show in Theorem \ref{theo3.2} of Section 3 a simple upper bound based on a contour argument (recall that $0 \le h_*$, see (\ref{0.3})):
\begin{equation}\label{0.9}
\limsup\limits_N \; \mbox{\f $\dis\frac{1}{N^{d-2}}$} \; \log \IP[A_N] \le - \mbox{\f $\dis\frac{1}{2d}$} \; \alpha^2 {\rm cap}_{\IR^d}([-1,1]^d), \; \mbox{for $\alpha \le 0$}.
\end{equation}

\n
This upper bound does not match the lower bound (\ref{0.8}), unless $h_* = h_{**} = 0$. But, as mentioned above, $h_*$ is known to be positive when $d$ is large, and actually expected to be positive for all $d \ge 3$ (see \cite{Mari07} for simulations when $d = 3$).

\medskip
Our main result comes in Section 5. We show in Theorem \ref{theo5.5} that
\begin{equation}\label{0.10}
\limsup\limits_N \; \mbox{\f $\dis\frac{1}{N^{d-2}}$} \; \log \IP[A_N] \le - \mbox{\f $\dis\frac{1}{2d}$} \; (\overline{h} - \alpha)^2\; {\rm cap}_{\IR^d}([-1,1]^d), \; \mbox{for $\alpha \le \ov{h}$},
\end{equation}

\n
where $\ov{h}$ is a certain critical value, see (\ref{5.3}) for the precise definition, such that $\alpha < \ov{h}$ corresponds to a strongly percolative regime for $E^{\ge \alpha}$. This critical value is similar to the critical value introduced in Section 2 \S 4 of \cite{DrewRathSapo14b}. It is plausible that $\ov{h} = h_{**}$, and (\ref{0.8}) and (\ref{0.10}) may actually offer matching upper and lower bounds. However, apart from the finiteness of $\ov{h}$ and the inequality $\ov{h} \le h_{*}$ ($\le h_{**}$), little is known at present about the actual value of $\ov{h}$ (not even that $\ov{h} \ge 0$).

\medskip
For this reason, we derive in Section 6 a variant of our main upper bound of Section 5, which borrows some techniques of \cite{RodrSzni13} developed for the proof of the positivity of $h_*$ for large $d$. We show in Theorem \ref{theo6.2} that there is an $h_0 > 0$ and a $d_0 \ge 3$ such that for $d \ge d_0$,
\begin{equation}\label{0.11}
\limsup\limits_N \; \mbox{\f $\dis\frac{1}{N^{d-2}}$} \; \log \IP[A_N] \le - \mbox{\f $\dis\frac{1}{2d}$} \; (h_0 - \alpha)^2\; {\rm cap}_{\IR^d}([-1,1]^d), \; \mbox{for $\alpha < h_0$}.
\end{equation}

\n
In particular, this last inequality shows that when $d \ge d_0$ the upper bound (\ref{0.9}) does not capture the correct exponential decay of $\IP[A_N]$.

\medskip
As a by-product of our results, we also derive some upper bounds on disconnection by random interlacements, or by simple random walk. A recent version due to Lupu \cite{Lupu} of the isomorphism theorem relating occupation-times of random interlacements to the Gaussian free field, see \cite{Szni12b}, provides a tool to transfer the upper bounds (\ref{0.10}), (\ref{0.11}) to random interlacements. If $\cI^u$ stands for the random interlacement at level $u \ge 0$, and $\cV^u = \IZ^d \backslash \cI^u$ for the vacant set at level $u$, we show in Theorem \ref{theo7.1} that if $\ov{h} > 0$ (with similar notation as in (\ref{0.6}), (\ref{0.10}))
\begin{equation}\label{0.12}
\begin{split}
\limsup\limits_N \; \mbox{\f $\dis\frac{1}{N^{d-2}}$} \; \log \IP[\partial B_N \;\overset{^{\mbox{\footnotesize $\cV^u$}}}{\mbox{\Large $\longleftrightarrow$}} \hspace{-4ex}/ \quad \;S_N] \le & - \mbox{\f $\dis\frac{1}{d}$} \; \Big(\sqrt{\mbox{\f $\dis\frac{\ov{h}\,^2}{2}$}} - \sqrt{u}\Big)^2 \,{\rm cap}_{\IR^d} ([-1,1]^d),
\\
&\;\; \mbox{for $u < \mbox{\f $\dis\frac{\ov{h}\,^2}{2}$}$},
\end{split}
\end{equation}
and a similar bound holds with $h_0$ in place of $\ov{h}$, when $d \ge d_0$ and $u < \frac{h^2_0}{2}$.

\medskip
In the case of the simple random walk starting from the origin, we denote by $\cV$ the complement of the set of sites visited by the walk, and by $P_0$ the canonical law. There is a natural coupling of $\cV$ under $P_0$ and $\cV^u$ under $\IP[\cdot | 0 \in \cI^u]$ ensuring that $\cV^u \subseteq \cV$, and as a result, we show in Corollary \ref{cor7.3} that if $\ov{h} > 0$
\begin{equation}\label{0.13a}
\limsup\limits_N \; \mbox{\f $\dis\frac{1}{N^{d-2}}$} \; \log \IP[\partial B_N \;\overset{^{\mbox{\footnotesize $\cV$}}}{\mbox{\Large $\longleftrightarrow$}} \hspace{-4ex}/ \quad \;S_N] \le  - \mbox{\f $\dis\frac{1}{2d}$} \; \ov{h}\,^{\!2} \,{\rm cap}_{\IR^d} ([-1,1]^d),
\end{equation}
and a similar bound holds with $h_0$ in place of $\ov{h}$, when $d \ge d_0$.

\medskip
In particular, when $d \ge d_0$, Theorem \ref{theo7.1} and Corollary \ref{cor7.3} establish an exponential decay at rate $N^{d-2}$ for the disconnection probability that appears in (\ref{0.12}), for small positive $u$, and for the probability that appears in (\ref{0.13a}). As far as we know, even these coarse bounds are new. We also refer to (\ref{1.10}) and Theorem 6.1 of \cite{Wind08b} on a related problem for the simple random walk.

\medskip
It is maybe helpful at this point to give some intuition about the bounds we obtain. The picture behind the lower bound (\ref{0.8}) and its proof in Theorem \ref{theo2.1} is roughly that when $\alpha < h_{**}$, to induce the disconnection event $A_N$, the Gaussian free field lowers itself by an amount slightly bigger than $h_{**} - \alpha$ in a neighborhood of $B_N$, so that for this shifted random field, being above level $\alpha$, amounts to being slightly above level $h_{**}$ for the Gaussian free field, and this entails a strongly non-percolative regime of the corresponding excursion-set. The feature that the upper bound (\ref{0.10}) may actually match the lower bound (\ref{0.8}) hints at a phenomenon of {\it entropic repulsion}. In a way, the situation is reminiscent of what happens when one considers the probability that $\varphi$ remains non-negative over $B_N$. As shown in Theorem 1.1 of \cite{BoltDeusZeit95}, see also Theorem 3.1 of \cite{Giac03}, one has the asymptotic behavior
\begin{equation}\label{0.13}
\begin{split}
\lim\limits_N \; \mbox{\f $\dis\frac{1}{N^{d-2} \log N}$} \; \log \IP[\varphi_x \ge 0, \; \mbox{for all} \; x \in B_N] = - \mbox{\f $\dis\frac{4}{2d}$} \; g(0,0) \,{\rm cap}_{\IR^d}([-1,1]^d),
\end{split}
\end{equation}
and one can show that the random field ``tends to shift upwards over $B_N$'' at a height of order $\sqrt{4 g(0,0) \log N}$, when $\varphi$ remains non-negative over $B_N$. In the situation we consider in the present work, if the equalities $\ov{h} = h_* = h_{**}$ hold, the intuitive picture should be as follows. When one enforces the disconnection of $\partial B_N$ from $S_N$ by the excursion-set of the Gaussian free field below $\alpha$, the effect for large $N$ on the field, say, close to the origin, depends on $\alpha$. If $\alpha > h_*$, the disconnection has a negligible effect. On the other hand, if $\alpha < h_*$, the disconnection induces ``a downward shift of the field by an amount $-(h_* - \alpha)$''.   

\medskip
We will now present a rough outline of the proof of the main Theorem \ref{theo5.5}, where (\ref{0.10}) is established. Unlike what would happen in the supercritical phase of Bernoulli percolation, where disconnection of the macroscopic box $B_N$ would involve an exponential cost of order $N^{d-1}$ (and surface tension), we are looking here for an exponential cost of order $N^{d-2}$. A quite substantial coarse graining takes place in the proof and roughly goes as follows. One considers ``columns'' of boxes of side-length $L$ (chosen of order $(N \log N)^{\frac{1}{d-1}})$ going from the surface $\partial B_N$ of $B_N$ to the surface $S_N$ of $B_{MN}$ (for simplicity assume $M=2$). The number of such columns has roughly order $(\frac{N}{L})^{d-1} = \frac{N^{d-2}}{\log N}$. For each box $B$ sitting in one of these columns, one decomposes the Gaussian free field $\varphi$ in a concentric box with side-length a large multiple of $L$, into an harmonic average denoted by $h_B$, and a local field denoted by $\psi_B$. The local fields $\psi_B$ enjoy good independence properties, see Lemma \ref{lem5.3}. In a first step, one shows in Proposition \ref{prop5.4} that in ``almost all columns'' all boxes have a local field, which ``percolates well'' above a level slightly below the critical value $\ov{h}$, except on an event with super-exponentially decaying probability (with respect to the rate $N^{d-2}$). On the other hand, when disconnection at level $\alpha$ occurs (i.e. $A_N$ is realized), each column must be blocked for percolation in the excursion set $E^{\ge \alpha}$. This forces the existence in most columns of a box where the harmonic average field reaches values  essentially smaller than $\alpha - \ov{h} (< 0)$. After a selection of these bad boxes, a step with not too high combinatorial complexity, thanks to our choice of $L$, we can use the Gaussian estimates developed in Section 4 (see Corollary \ref{cor4.4}) to bound $\IP[A_N]$ in essence by $\exp\{-\frac{1}{2} \,(\ov{h} - \alpha)^2 \inf_{\cC}\, {\rm cap}(C) + o(N^{d-2})\}$, where ${\rm cap}(C)$ stands for the simple random walk capacity of $C = \bigcup_{B \in \cC} B$ and $\cC$ runs over the various collections of selected bad boxes. Using a projection on the surface of $B_N$ and a Wiener-type criterion, one gets an asymptotic lower bound on ${\rm cap}(C)$ in terms of ${\rm cap}(B_N)$, uniformly over $\cC$, and the upper bound (\ref{0.10}) quickly follows.

\medskip
Theorem \ref{theo6.2}, where (\ref{0.11}) is proven, follows the same strategy, but uses a different scheme for percolation in each column, along a thick two-dimensional slab. This procedure is developed in Theorem \ref{theo6.1} and uses techniques of Section 3 of \cite{RodrSzni13}. The Gaussian free field restricted to the slab can be viewed as a ``small perturbation'' of an i.i.d. field, when $d$ is large, and methods of Bernoulli percolation such as static renormalization, see Chapter 7 \S 4 of \cite{Grim99}, can be brought to bear.

\medskip
We will now describe the organization of this article. Section 1 introduces further notation and recalls various facts concerning random walk, potential theory, the Gaussian free field, and its level-set percolation. In the short Section 2 we prove the lower bound (\ref{0.8}), see Theorem \ref{theo2.1}. In Section 3 we employ a contour argument in the spirit of \cite{BricLeboMaes87} and prove (\ref{0.9}) in Theorem \ref{theo3.2}. In Section 4 we develop Gaussian bounds, which prepare the ground for the proofs of our main upper bounds in the next two sections, see Theorem \ref{theo4.2} and Corollary \ref{cor4.4}. Section 5 contains the proof of (\ref{0.10}) in the main Theorem \ref{theo5.5}. In Section 6 we prove (\ref{0.11}) in Theorem \ref{theo6.2}, using an adaptation of the proof of Theorem \ref{theo5.5}. The percolative estimates are developed in Theorem \ref{theo6.1}. In Section 7 we apply the upper bounds of Sections 5 and 6 to the derivation of (\ref{0.12}) in the context of random interlacements, and (\ref{0.13a}) in the context of simple random walk, see Theorem \ref{theo7.1} and Corollary \ref{cor7.3}.

\medskip
Finally, let us state the convention we use concerning constants. We denote by $c,c',\ov{c}$ positive constants changing from place to place, which simply depend on $d$. Numbered constants such as $c_0,c_1,\dots$ refer to the value corresponding to their first appearance in the text. Dependence of constants on additional parameters appears in the notation.

\bigskip\n
{\bf Acknowledgements:} We wish to thank Ofer Zeitouni for stimulating conversations and the Weizmann Institute of Science for its hospitality and support via a Joseph Meyerhoff Visiting Professorship.

\section{Notation and some useful facts}
\setcounter{equation}{0}

In this section we introduce some further notation and review some classical facts concerning random walks, potential theory, the Gaussian free field, the percolative properties of its level-sets, and an entropy inequality entering the change of probability method. These various ingredients will be useful in the subsequent sections.

\medskip
We begin with some notation. Given real numbers $s,t$, we write $s \wedge t$ and $s \vee t$ for the minimum and the maximum of $s$ and $t$, and denote by $[s]$ the integer part of $s$, when $s \ge 0$. We write $| \cdot |$ and $| \cdot |_\infty$ for the Euclidean and the $\ell^\infty$-norms on $\IR^d$. We tacitly assume throughout the article that $d \ge 3$. Given $x \in \IZ^d$ and $r \ge 0$, we let $B(x,r) = \{y \in \IZ^d$; $|y - x|_\infty \le r\}$ stand for the closed $\ell^\infty$-ball of radius $r$ around $x$. Given $K, K'$ subsets of $\IZ^d$, we denote by $d(K,K') = \inf\{|x - x'|_\infty$; $x \in K, x'\in K'\}$ the mutual $\ell^\infty$-distance between $K$ and $K'$. When $K = \{x\}$, we simply write $d(x,K')$. We write ${\rm diam}(K) = \sup\{|x-y|_\infty$; $x, y \in K\}$ for the $\ell^\infty$-diameter of $K$, $|K|$ for the cardinality of $K$, and $K \subset \subset \IZ^d$, to state that $K$ is a finite subset of $\IZ^d$. We denote by $\partial K = \{y \in \IZ^d \backslash K$; $\exists x \in K$, $|y-x| = 1\}$ the boundary of $K$, and by $\partial_i K = \{x \in K$; $\exists y \in \IZ^d \backslash K$, $|x - y| = 1\}$ the internal boundary of $K$.

\medskip
We say that $x,y$ in $\IZ^d$ are neighbors, when $|x - y| = 1$, and sometimes write $x \sim y$. We say that they are $*$-neighbors when $|x-y|_\infty = 1$. We call $\pi: \{0,\dots,n\} \rightarrow \IZ^d$ a path when $\pi(i) \sim \pi(i-1)$, for $1 \le i \le n$. We define a $*$-path accordingly. Given $K,L,U$ subsets of $\IZ^d$, we say that $K$ and $L$ are connected by $U$ and write $K \stackrel{U}{\longleftrightarrow} L$, when there exists a path with values in $U$ ($\subseteq \IZ^d$), which starts in $K$ and ends in $L$. Otherwise we say that $K$ and $L$ are not connected by $U$, and write $K\stackrel{^U}{\longleftrightarrow}\hspace{-3.5ex}/ \quad L$ (see for instance (\ref{0.12})).

\medskip
We now turn to discrete time simple random walk on $\IZ^d$. we denote by $(X_n)_{n \ge 0}$ the canonical process on $(\IZ^d)^\IN$ and by $P_x$ the canonical law starting from $x$. We write $E_x$ for the corresponding expectation. When $\rho$ is a measure on $\IZ^d$, we write $P_\rho = \sum_x \rho(x) P_x$ (not necessarily a probability measure) and $E_\rho$ for the corresponding ``expectation'' (i.e. the integral with respect to $P_\rho$). Given $U \subseteq \IZ^d$, we denote by $H_U = \inf\{ n \ge 0; X_n \in U\}$ the entrance time in $U$, by $\wt{H}_U = \inf\{n \ge 1, X_n \in U\}$ the hitting time of $U$, and by $T_U = \inf\{n \ge 0, X_n \notin U\}$ the exit time from $U$.

\medskip
We let $g(\cdot,\cdot)$ stand for the Green function of the walk,
\begin{equation}\label{1.1}
g(x,y) = \dis\sum\limits_{n \ge 0} P_x [X_n = y], \;\mbox{for $x,y \in \IZ^d$}.
\end{equation}

\n
Since $d \ge 3$, the Green function is finite. Due to translation invariance, one has $g(x,y) = g(x-y,0) \stackrel{{\rm def}}{=} g(x-y)$, and one knows that (see Theorem 5.4, p.~31 of \cite{Lawl91})
\begin{equation}\label{1.2}
g(x) \sim c_0 |x|^{2-d}, \;\mbox{as $|x| \r \infty$, with $c_0 = \mbox{\f $\dis\frac{d}{2}$} \;\Gamma \Big(\dis\mbox{\f $\dis\frac{d}{2}$} - 1\Big) \;\mbox{\f $\dis\frac{1}{\pi^{\frac{d}{2}}}$}$}\,.
\end{equation}
Given $U \subseteq \IZ^d$, the Green function killed outside $U$ is defined as
\begin{equation}\label{1.3}
g_U(x,y) = \dis\sum\limits_{n \ge 0} P_x [X_n = y, n < T_U]\,.
\end{equation}
It is a symmetric function of $x$ and $y$, which vanishes if $x \notin U$ or $y \notin U$. As a direct application of the strong Markov property at the exit time of $U$, we have
\begin{equation}\label{1.4}
g(x,y) = g_U (x,y) + E_x[T_U < \infty, g(X_{T_U},y)], \; \mbox{for} \; x,y \in \IZ^d.\end{equation}

\n
We then discuss some potential theory attached to simple random walk. Given $K \subset \subset \IZ^d$, we write $e_K$ for the equilibrium measure of $K$:
\begin{equation}\label{1.5}
e_K(x) = P_x [\wt{H}_K = \infty] \,1_K (x), \; \mbox{for $x \in \IZ^d$}
\end{equation}
(it is supported by the internal boundary of $K$), and ${\rm cap}(K)$ for the capacity of $K$, which is the total mass of $e_K$:
\begin{equation}\label{1.6}
{\rm cap}(K) = \dis\sum\limits_{x \in K} e_K(x).
\end{equation}
In the special case of the $|\cdot|_\infty$-ball $B_L (= B(0,L))$, one knows that (see (2.16), p.~53 of \cite{Lawl91})
\begin{equation}\label{1.7}
c L^{d-2} \le {\rm cap}(B_L) \le c' L^{d-2}, \; \mbox{for $L \ge 1$}.\end{equation}

\n
One also has (see for instance (2.10) on p.~18 of \cite{Woes00}) a characterization of the capacity through the Dirichlet form:
\begin{equation}\label{1.8}
{\rm cap}(K) = \inf\limits_f \, \cE(f,f),
\end{equation}
where $f$ varies over the set of finitely supported functions on $\IZ^d$ with value at least $1$ on $K$, and for $g$: $\IZ^d \r \IR$
\begin{equation}\label{1.9}
\cE(g,g) = \fr \;\dis\sum\limits_{x \sim y} \;   \mbox{\f $\dis\frac{1}{2d}$} \;\big(g(y) - g(x)\big)^2
\end{equation}
(one also defines $\cE(f,g) = \frac{1}{2}\;  \mbox{\footnotesize $\sum_{x \sim y}\; \frac{1}{2d} (f(y) - f(x)) (g(y) - g(x))$}$ by polarization of the above formula when the resulting series is absolutely convergent).

\medskip
One has a further variational characterization of the capacity, which is convenient for the derivation of lower bounds on the capacity
\begin{equation}\label{1.10}
{\rm cap}(K) = \{ \inf\limits_{\nu} E(\nu)\}^{-1}, \;\mbox{with} \; E(\nu) = \dis\sum\limits_{x,y} \nu(x) \,\nu(y) \,g(x,y),
\end{equation}
and the infimum runs over probability measures supported on $K$.

\medskip
We also recall (see for instance Theorem T.1, p.~300 of \cite{Spit01}) that the entrance probability in $K$ can be expressed in terms of the Green function and the equilibrium measure as
\begin{equation}\label{1.11}
P_x[H_K < \infty] = \dis\sum\limits_{y \in K} g(x,y) \,e_K(y), \;\mbox{for} \; x \in \IZ^d.
\end{equation}
Further, one has the sweeping identity for $K \subset K' \subset \subset \IZ^d$
\begin{equation}\label{1.12}
e_K(y) = P_{e_{K'}} [H_K < \infty, X_{H_K} = y], \;\mbox{for all $y \in \IZ^d$}
\end{equation}
(see for instance (1.18) of \cite{LiSzni14}, and we recall the notation above (\ref{1.1}) for $P_{e_{K'}}$). 

\medskip
We now turn to the Gaussian free field on $\IZ^d$. We denote by $\IP$ its canonical law on $\IR^{\IZ^d}$ and by $\varphi = (\varphi_x)_{x \in \IZ^d}$ the canonical random field, so that under $\IP$, $\varphi$ is a centered Gaussian field with covariance
\begin{equation}\label{1.13}
\IE[\varphi_x \varphi_y] = g(x,y), \; \mbox{for $x,y \in \IZ^d$}.
\end{equation}

\n
Further, for finitely supported $f$ one has (see below (\ref{1.9}) for notation)
\begin{equation}\label{1.14}
\left\{ \begin{array}{rl}
{\rm i)} & \quad E[\cE(f,\varphi) \,\varphi_x] = f_x, \; \mbox{for $x \in \IZ^d$}
\\[1ex]
{\rm ii)} & \quad E[\cE(f,\varphi)^2] = \cE(f,f).
\end{array}\right.
\end{equation}

\medskip\n
(the second identity readily follows from (\ref{1.14})~i) and the Gauss-Green identity $\cE(f,g) =  \mbox{\footnotesize $- \sum_{x \in \IZ^d} \Delta f(x) \,g(x)$}$, for $f,g$ functions on $\IZ^d$, with $f$ finitely supported, and $\Delta f(x) = \linebreak \frac{1}{2d}  \mbox{\footnotesize $\sum_{y \sim x} (f(y) - f(x))$}$. The first identity (\ref{1.14})~i) follows from (\ref{1.13}) and the Gauss-Green identity).

\medskip
Given $U \subset \subset \IZ^d$ (i.e. $U$ finite subset of $\IZ^d$, in the notation at the beginning of this section), one defines the random fields
\begin{align}
&h^U_x = E_x[\varphi_{X_{T_U}}] \Big(= \dis\sum\limits_{y \in \IZ^d} P_x[X_{T_U} = y] \,\varphi_y\Big) , \; \mbox{for $x \in \IZ^d$} \label{1.15}
\\
&\mbox{(note that $h^U_x = \varphi_x$ for $x \in  \IZ^d \backslash U$), and}  \nonumber
\\[2ex]
&\psi^U_x = \varphi_x - h^U_x, \;\mbox{for $x \in \IZ^d$} \label{1.16} 
\\
&\mbox{(note that $\psi^U_x = 0$, when $x \in \IZ^d \backslash U$)}. \nonumber
\end{align}

\n
We will sometimes refer to $h^U$ as the harmonic average of $\varphi$ in $U$, and to $\psi^U$ as the local field in $U$. One thus has the decomposition
\begin{equation}\label{1.17}
\varphi_x = h^U_x + \psi^U_x, \;x \in \IZ^d,
\end{equation}

\n
and the Markov property of the Gaussian free field can be expressed as the fact (see for instance Lemma 1.2 of \cite{RodrSzni13}):
\begin{equation}\label{1.18}
\begin{array}{l}
\mbox{$(\psi^U_x)_{x \in \IZ^d}$ is independent of $\sigma(\varphi_y, y \in U^c)$ (in particular it is independent from} 
\\
\mbox{$(h^U_x)_{x \in \IZ^d})$ and distributed as a centered Gaussian field with covariance $g_U(\cdot,\cdot)$}.
\end{array}
\end{equation}

\medskip\n
Hence, the conditional law of $(\varphi_x)_{x \in U}$ given $\sigma(\varphi_y,y \in U^c)$ only depends on $(h^U_x)_{x \in U}$, which is $\sigma(\varphi_y,y \in \partial U)$-measurable, see (\ref{1.15}).

\medskip
Level-set percolation of the Gaussian free field was discussed in the Introduction. We simply recall here the following feature of the strongly non-percolative regime for $E^{\ge \alpha}$. We remind that $h_{**}$ has been defined in (\ref{0.5}) (see also Theorem 2.6 of \cite{RodrSzni13} and Theorem 2.1 of \cite{PopoRath}) and that $0 \le h_* \le h_{**} < \infty$. One also knows (see above references) that
\begin{equation}\label{1.19}
\mbox{for} \; \alpha > h_{**}, \;\IP[0 \stackrel{\ge \alpha}{\longleftrightarrow} \partial B_L] \le c_1(\alpha) \,e^{-c_2(\alpha)L^{c_3}}, \;\mbox{for $L \ge 0$}
\end{equation}

\n
(actually when $d \ge 4$ one can choose $c_3 = 1$, and when $d = 3$, $c_3 = \frac{1}{2}$ or any value in $(0,1)$, see \cite{PopoRath}).

\medskip
As a direct application of a union bound and (\ref{1.19}) one finds that, in the notation of (\ref{0.6}), when $\alpha > h_{**}$, $\IP[\partial B_N \stackrel{\ge \alpha}{\longleftrightarrow} S_N] \underset{N}{\longrightarrow} 0$. We thus see that
\begin{equation}\label{1.20}
\IP[A_N] \underset{N}{\longrightarrow} 1, \;\mbox{when $\alpha > h_{**}$}.
\end{equation}

\n
Finally, we recall a classical inequality concerning the relative entropy, which will be useful in the next section. For $\wt{\IP}$ absolutely continuous with respect to $\IP$, the relative entropy of $\wt{\IP}$ with respect to $\IP$ is defined as
\begin{equation}\label{1.21}
H(\wt{\IP}|\IP) = \wt{\IE} \Big[\log \mbox{\f $\dis\frac{d \wt{\IP}}{d\IP}$} \Big] = \IE \Big[\mbox{\f $\dis\frac{d\wt{\IP}}{d\IP}$} \;\log\mbox{\f $\dis\frac{d\wt{\IP}}{d\IP}$}\Big] \in [0,\infty],
\end{equation}
where $\wt{\IE}$ stands for the expectation with respect to $\wt{\IP}$.

\medskip
Given an event $A$ with the positive $\wt{\IP}$-probability one has (see p.~76 of \cite{DeusStro89}):
\begin{equation}\label{1.22}
\IP[A] \ge \wt{\IP}[A] \,e^{-\frac{1}{\wt{\IP}[A]} (H(\wt{\IP}|\IP) +\frac{1}{e})}.
\end{equation}

\section{Disconnection lower bound}
\setcounter{equation}{0}

We have seen in (\ref{1.20}) that when the level $\alpha$ is bigger than $h_{**}$ the event $A_N$ (see (\ref{0.6}) or (\ref{2.1}) below) becomes typical for large $N$. In this section, we instead consider a level $\alpha \le h_{**}$, and derive in Theorem \ref{theo2.1} an asymptotic lower bound on the probability of $A_N$. This is very much in the spirit of the lower  bound obtained in Theorem 0.1 of \cite{LiSzni14} in the context of random interlacements, see also Section 7 below. The situation is however substantially simpler for the level-set percolation of the Gaussian free field: the method of change of probability merely involves a deterministic shift of the Gaussian free field (in \cite{LiSzni14} one needed the so-called ``tilted interlacements''). One also has here an additional simplifying feature: in the present work we simply discuss the disconnection of the discrete blow-up $B_N$ of $[-1,1]^d$ (in \cite{LiSzni14} one considered the discrete blow-up of a general compact subset of $\IR^d$).

\medskip
We recall that $M > 1$ is some given real number, and the basic disconnection event at level $\alpha$ (in $\IR$) is
\begin{equation}\label{2.1}
A_N = \{\partial B_N  \stackrel{\ge \alpha}{\mbox{\Large $\longleftrightarrow$}}\hspace{-4ex}\mbox{\footnotesize $/$} \quad \;S_N\} \;\;\mbox{(with $S_N = \{x \in \IZ^d; \,|x|_\infty = [MN]\}$)}.
\end{equation}

\n
Note that when $MN \ge N+1$ the event $A_N$ increases with $M$ (or perhaps more directly, the complementary event decreases with $M$).

\medskip
We recall that ${\rm cap}_{\IR^d}(\cdot)$ stands for the Brownian capacity (see for instance \cite{PortSton78}, p.~58). The main result of this section is the following
\begin{theorem}\label{theo2.1}
Assume that $\alpha \le h_{**}$, then one has
\begin{equation}\label{2.2}
\liminf_N \;\mbox{\f $\dis\frac{1}{N^{d-2}}$} \; \log \IP[A_N] \ge - \mbox{\f $\dis\frac{1}{2d}$} \;(h_{**} - \alpha)^2 \;{\rm cap}_{\IR^d}([-1,1]^d).
\end{equation}
\end{theorem}

\begin{proof}
We use the method of change of probability and the relative entropy inequality (\ref{1.22}). Given $f$: $\IZ^d \r \IR$, finitely supported, we introduce the probability (on $\IR^{\IZ^d}$):
\begin{equation}\label{2.3}
\wt{\IP} = \exp\Big\{\cE(f,\varphi) - \fr \;\cE(f,f)\Big\} \,\IP .
\end{equation}

\n
By (\ref{1.14}) and Cameron-Martin's formula, $\wt{\IP}$ is indeed a probability measure and
\begin{equation}\label{2.4}
\mbox{$\varphi$ under $\wt{\IP}$ has the same law as $(\varphi_x + f_x)_{x \in \IZ^d}$ under $\IP$.}
\end{equation}

\n
We now choose $\varepsilon, \eta > 0$ with $ 1 + \eta < M$, and a smooth function $g$ compactly supported in $(-M, M)^d$ and smaller or equal to $-(h_{**} - \alpha + \ve)$ on $[-(1+\eta), 1 + \eta]^d$. We define the sequence of finitely supported functions on $\IZ^d$:
\begin{equation}\label{2.5}
f_N(x) = g\Big(\mbox{\f $\dis\frac{x}{N}$}\Big), \;\mbox{for $x \in \IZ^d$, $N \ge 1$}.
\end{equation}
We denote by $\wt{\IP}_N$ the probability attached to $f_N$ (in place of $f$) by (\ref{2.3}). By the entropy inequality (\ref{1.22}), we know that
\begin{equation}\label{2.6}
\IP[A_N] \ge \wt{\IP}_N [A_N] \,\exp\Big\{- \mbox{\f $\dis\frac{1}{\wt{\IP}_N [A_N]}$} \Big(H(\wt{\IP}_N | \IP) + \mbox{\f $\dis\frac{1}{e}$}\Big)\Big\}.
\end{equation}
In addition, we find that
\begin{equation}\label{2.7}
H(\wt{\IP}_N | \IP) \underset{(\ref{2.3})}{\stackrel{(\ref{1.21})}{=}} \wt{\IE} [ \cE(f_N,\varphi)] - \fr \; \cE(f_N,f_N) \stackrel{(\ref{2.4})}{=} \fr \;\cE (f_N,f_N).
\end{equation}
The next step is to prove that
\begin{equation}\label{2.8}
\lim\limits_N \wt{\IP} [A_N] = 1.
\end{equation}
Indeed, by (\ref{2.4}), with hopefully obvious notation, we see that
\begin{equation}\label{2.9}
\wt{\IP}_N[A_N] = \IP[\partial B_N \;  \overset{^{\mbox{\footnotesize $\ge \alpha - f_N$}}}{\mbox{\Large $\longleftrightarrow$}} \hspace{-4.5ex}/ \quad \; S_N]
\end{equation}

\n
Since $f_N \le -(h_{**} - \alpha + \ve)$ on $B_{N(1 + \eta)}$, defining $\wt{S}_N = \{x \in \IZ^d; |x|_\infty = [(1 + \eta)N]\}$, we see that for large $N$ the probability in the right-hand side of (\ref{2.9}) is bigger or equal to
\begin{equation*}
\IP[\partial B_N  \overset{^{\mbox{\footnotesize $\ge h_{**} + \ve$}}}{\mbox{\Large $\longleftrightarrow$}} \hspace{-4.5ex}/ \quad \; \wt{S}_N] \underset{N}{\longrightarrow} 1 \;\; \mbox{(by (\ref{1.20}) with $M = 1 + \eta$ and $\alpha = h_{**} + \ve$)}.
\end{equation*}
This proves (\ref{2.8}).

\medskip
Coming back to (\ref{2.6}), we now find by (\ref{2.7}), (\ref{2.8}) that
\begin{equation}\label{2.10}
\liminf\limits_N \; \mbox{\f $\dis\frac{1}{N^{d-2}}$} \log \IP[A_N] \ge - \limsup\limits_N \; \mbox{\f $\dis\frac{1}{N^{d-2}}$} \; \fr \; \cE (f_N ,f_N).
\end{equation}
On the other hand, by (\ref{1.9}) and the choice of $f_N$ in (\ref{2.5}), we have
\begin{equation*}
\dis\frac{1}{N^{d-2}} \; \cE(f_N, f_N) = \mbox{\f $\dis\frac{1}{4 d N^d}$} \; \dis\sum\limits_{x \in \IZ^d} \; \dis\sum\limits_{|e| = 1} N^2 \Big(g\Big(\mbox{\f $\dis\frac{x + e}{N}$}\Big) - g\Big(\mbox{\f $\dis\frac{x}{N}$}\Big)\Big)^2\,.
\end{equation*}

\n
It now follows from the smoothness of $g$ and a Riemann sum argument that we have
\begin{equation}\label{2.11}
\lim\limits_N \; \mbox{\f $\dis\frac{1}{N^{d-2}}$} \; \cE(f_N,f_N) = \mbox{\f $\dis\frac{1}{2d}$} \; \dis\int |\nabla g(y)|^2 dy = \mbox{\f $\dis\frac{1}{d}$} \; \cE_{\IR^d} (g,g),
\end{equation}

\n
where $\cE_{\IR^d}(\cdot,\cdot)$ stands for the Dirichlet form attached to Brownian motion (and Lebesgue measure on $\IR^d$). Thus, optimizing over $g$ (see Lemma 2.2.7, p.~80 of \cite{FukuOshiTake94}, or below (2.28) of \cite{LiSzni14} for a very similar argument) that
\begin{equation}\label{2.12}
\liminf\limits_N \;\mbox{\f $\dis\frac{1}{N^{d-2}}$} \;\log \IP[A_N] \ge - \mbox{\f $\dis\frac{1}{2d}$} \;(h_{**} - \alpha + \ve)^2\, {\rm cap}_{\IR^d} ([-(1 + \eta), 1 + \eta]^d). 
\end{equation}
Letting $\eta$ and $\ve$ tend to $0$, we obtain the claim (\ref{2.2}).
\end{proof}

\begin{remark}\label{rem2.2} \rm
If we denote by $\{\partial B_N \overset{^{\mbox{\footnotesize $\ge \alpha$}}}{\longleftrightarrow} \hspace{-3.5ex}/ \quad \infty\}$ the event where there is no infinite nearest neighbor path in $E^{\ge \alpha}$ starting from $\partial B_N$, it is plain that this event contains $A_N$ when $MN \ge N+1$. Hence, as an immediate consequence of Theorem \ref{theo2.1}, we see that
\begin{equation}\label{2.13}
\liminf\limits_N \; \mbox{\f $\dis\frac{1}{N^{d-2}}$} \log \IP[\partial B_N \overset{^{\mbox{\footnotesize $\ge \alpha$}}}{\mbox{\Large $\longleftrightarrow$}} \hspace{-3.8ex}/ \quad \;\; \infty] \ge - \mbox{\f $\dis\frac{1}{2d}$} \;(h_{**} - \alpha)^2 {\rm cap}_{\IR^d} ([-1,1]^d).
\end{equation}

\hfill $\square$
\end{remark}

\section{A disconnection upper bound based on a contour argument}
\setcounter{equation}{0}

In this section we derive an upper bound on the probability of the disconnection event $A_N$, see (\ref{0.6}) or (\ref{2.1}), when the level $\alpha$ is negative. We use an argument in the spirit of the proof of Theorem 2 of \cite{BricLeboMaes87}, based on the notion of maximal contour surrounding $B_N$ in $B_{MN}$, where the Gaussian free field lies below $\alpha$. In this fashion, we use a form of strong Markov property of the Gaussian free field. Our main result Theorem \ref{theo3.2} yields both a quantitative upper bound, cf.~(\ref{3.6}), and an asymptotic upper bound, cf.~(\ref{3.7}). This asymptotic upper bound does not match the asymptotic lower bound from Theorem \ref{theo2.1} in the previous section, when $h_{**} > 0$ (and $h_* \le h_{**}$ is known to be positive for large $d$, cf.~Theorem 3.3 of \cite{RodrSzni13}). In the next sections we will aim at improving this defect.

\medskip
We first introduce some definitions concerning contours. We recall that we tacitly assume $d \ge 3$. Given $N \ge 0$, we say that $C \subseteq \IZ^d$ is a contour surrounding $B_N$, when there exists a finite connected subset $K$ of $\IZ^d$ containing $B_N$, such that $C = \partial K$. Given a contour $C$ surrounding $B_N$, the above set $K$ is uniquely determined (it is the connected component of $\IZ^d \backslash C$ containing $B_N$). We write $K = {\rm Int}\,C$.

\medskip
Given a finite family of contours $C_i, 1 \le i \le n$, surrounding $B_N$, we define the maximal contour via
\begin{equation}\label{3.1}
\max\{ C_1,\dots, C_n\} = \partial \Big(\bigcup\limits^n_{i=1} {\rm Int}\,C_i\Big),
\end{equation}
and we note that
\begin{equation}\label{3.2}
\max \{C_1,\dots,C_n\} \subseteq \bigcup\limits^n_{i=1} C_i .
\end{equation}
The next lemma relates the above notion of contour with the disconnection event $A_N$ (see (\ref{0.6}) or (\ref{2.1})).
\begin{lemma}\label{lem3.1} (recall $\alpha \in \IR$)

\medskip
Assume that $MN \ge N+1$, then
\begin{equation}\label{3.3}
\mbox{$A_N = \{\varphi$; there is a contour surrounding $B_N$ contained in $B_{MN}$, where $\varphi < \alpha\}$}.
\end{equation}
\end{lemma}

\begin{proof}
Denote by $\wt{A}_N$ the event on the right-hand side of (\ref{3.3}). We fist show that $\wt{A}_N \subseteq A_N$. Note that the interiors of the contours that appear in the definitions of $\wt{A}_N$ are necessarily contained in $B_{MN} \backslash S_N = \{x \in \IZ^d$; $|x|_\infty < [MN]\}$ (otherwise such an interior would contain all $x$ in $\IZ^d$ with $|x|_\infty > [MN]\}$ (otherwise such an interior would contain all $x$ in $\IZ^d$ with $|x|_\infty > [MN]$, and be infinite). Hence, on $\wt{A}_N$, any path from $B_N$ to $S_N$ must meet a contour $C$ surrounding $B_N$, where $\varphi < \alpha$, and therefore $A_N$ is realized. To prove the reverse inclusion $A_N \subseteq \wt{A}_N$, we argue as follows. For all $x \in \partial B_N$ we consider $C_{\ge \alpha} (x)$, the connected component of $x$ in $E^{\ge \alpha}$ (defined as the empty set when $\varphi_x < \alpha$). On $A_N$, the random connected subset $B_N \cup (\bigcup_{x \in \partial B_N} C_{\ge \alpha} (x)$) contains $B_N$ and does not intersect $S_N$ (by definition of $A_N$). Hence, its boundary is a contour surrounding $B_N$ and contained in $B_{MN}$, where $\varphi < \alpha$. This shows that $A_N \subseteq \wt{A}_N$ and completes the proof of the equality (\ref{3.3}).
\end{proof}

\medskip
By the above lemma, when $MN \ge N+1$, we can define on $A_N$
\begin{equation}\label{3.4}
\begin{split}
C_{< \alpha}^{\max} = &\; \mbox{the maximal contour in the family of contours surrounding $B_N$}
\\
& \;\mbox{contained in $B_{MN}$, where $\varphi < \alpha$}.
\end{split}
\end{equation}

\n
The crucial property (reminiscent of stopping times) satisfied by $C_{< \alpha}^{\max}$ is the following:
\begin{equation}\label{3.5}
\begin{array}{l}
\mbox{for any contour $C$ surrounding $B_N$ and contained in $B_{MN}$, the event}
\\
\mbox{$\{C_{< \alpha}^{\max} = C\}$ is $\sigma(\varphi_x, x \in U^c)$-measurable, with $U = {\rm Int} \,C$}.
\end{array}
\end{equation}

\medskip\n
Indeed, the above event is characterized by the fact that $\varphi < \alpha$ on $C$, and for any finite connected set $V \supsetneq U$ with $\partial V \subseteq B_{MN}$, one has $\varphi_x \ge \alpha$ for some $x \in \partial V$.

\medskip
We now come to the main result of this section.
\begin{theorem}\label{theo3.2} $(M > 1, \alpha < 0)$ 

\medskip
Assume $MN \ge N+1$. Then, in the notation of (\ref{2.1}) one has
\begin{equation}\label{3.6}
\IP[A_N] \le 2 \exp\Big\{ - \fr \;\alpha^2 {\rm cap}(B_N)\Big\}.
\end{equation}
Moreover, one has
\begin{equation}\label{3.7}
\limsup\limits_N \mbox{\f $\dis\frac{1}{N^{d-2}}$} \;\log \IP[A_N] \le - \mbox{\f $\dis\frac{1}{2d}$} \;\alpha^2 {\rm cap}_{\IR^d} ([-1,1]^d).
\end{equation}
\end{theorem}

\begin{proof}
We first prove (\ref{3.6}). By Lemma \ref{lem3.1} and (\ref{3.4}) we can partition the event $A_N$ according to the different possibilities for $C^{\max}_{< \alpha}$ and write
\begin{equation}\label{3.8}
A_N = \bigcup_C \{C_{< \alpha}^{\max} = C\} \;\; \mbox{(disjoint union)},
\end{equation}
where $C$ runs over the collection of contours surrounding $B_N$ and contained in $B_{MN}$. Given such a contour $C$, we write $U = {\rm Int} C( \supseteq B_N)$ and find by (\ref{1.17}), (\ref{1.18}) that on $\{C^{\max}_{< \alpha} = C\}$
\begin{equation}\label{3.9}
\varphi = \psi^U + h^U,
\end{equation}
where $\psi^U$ is independent of $\sigma(\varphi_y, y \in U^c)$, and a centered Gaussian process with covariance $g_U(\cdot,\cdot)$, and $h^U_x = E_x[\varphi_{X_{T_U}}]$, $x \in \IZ^d$, is $\sigma(\varphi_y,y \in U^c)$-measurable and harmonic on $U$.

\medskip
We denote by $\nu$ the equilibrium measure of $B_N$ and by $\ov{\nu} = \frac{\nu}{{\rm cap}(B_N)}$ the normalized equilibrium measure of $B_N$, see (\ref{1.5}), (\ref{1.6}). We thus find that
\begin{equation}\label{3.10}
\begin{array}{l}
\IP\Big[\dis\sum\limits_x \ov{\nu}(x) \,\varphi_x \le \alpha\Big] \ge \IP\Big[\dis\sum\limits_x \ov{\nu} (x) \,\varphi_x \le \alpha, A_N\Big] \stackrel{(\ref{3.8})}{=}
\\[1ex]
\dis\sum\limits_C \IP\Big[\dis\sum\limits_x \ov{\nu}(x) \,\varphi_x \le \alpha, C^{\max}_{< \alpha} = C\Big] \stackrel{(\ref{3.9})}{=}
\\[1ex]
\dis\sum\limits_C \IP\Big[\dis\sum\limits_x \ov{\nu}(x) \,(\psi_x^U + h^U_x) \le \alpha, C^{\max}_{< \alpha} = C\Big] \ge
\\[1ex]
\dis\sum\limits_C \IP\Big[\dis\sum\limits_x \ov{\nu}(x) \,\psi_x^U \le 0, \dis\sum\limits_x \ov{\nu}(x) \,h^U_x \le \alpha, C^{\max}_{< \alpha} = C\Big]  \underset{\rm below\;(\ref{3.9})}{\stackrel{(\ref{3.5})}{=}}
\\[1ex]
\dis\sum\limits_C \IP\Big[\dis\sum\limits_x \ov{\nu}(x) \,\psi_x^U \le 0\Big]\;\IP \Big[\dis\sum\limits_x \ov{\nu}(x) \,h^U_x \le \alpha, C^{\max}_{< \alpha} = C\Big]  .
\end{array}
\end{equation}

\n
Note that $h^U$ is harmonic on $U$ with values $< \alpha$ on $\partial U = C$, so the rightmost probability in the last line of (\ref{3.10}) equals $\IP[C_{< \alpha}^{\max} = C]$. Moreover, $\sum_x \ov{\nu}(x) \psi_x^U$ is a centered Gaussian variable so that $\IP[\sum_x \ov{\nu}(x) \psi^U_x \le 0] \ge \frac{1}{2}$ (actually, one has an equality because the above Gaussian variable is non-degenerate, as can easily been argued). Inserting these observations in the last line of (\ref{3.10}) we find that
\begin{equation}\label{3.11}
\IP\Big[ \dis\sum\limits_x \ov{\nu}(x) \varphi_x \le \alpha\Big] \ge \fr \; \dis\sum\limits_C \IP[C^{\max}_{< \alpha} = C] \stackrel{(\ref{3.8})}{=} \fr \;\IP[A_N].
\end{equation}
Further, $\sum_x \ov{\nu}(x) \varphi_x$ is a centered Gaussian variable with variance
\begin{equation}\label{3.12}
{\rm var}\Big(\dis\sum_x \ov{\nu}(x) \varphi_x\Big) \stackrel{(\ref{1.13})}{=} \dis\sum\limits_{x,y} \ov{\nu}(x) \,\ov{\nu}(y) \,g(x,y) \stackrel{(\ref{1.11})}{=} \mbox{\f $\dis\frac{1}{{\rm cap}(B_N)}$}.
\end{equation}
Hence, using a standard bound on the tail of a Gaussian variable, we find 
\begin{equation}\label{3.13}
\IP[A_N] \stackrel{(\ref{3.11})}{\le} 2 \IP \Big[\dis\sum\limits_x \ov{\nu}(x) \varphi_x \le \alpha\Big] \le 2 \exp \Big\{ - \fr \; {\rm cap} (B_N)\alpha^2\Big\}.
\end{equation}
This proves (\ref{3.6}). Further, one know that
\begin{equation}\label{3.14}
\lim\limits_N \; \mbox{\f $\dis\frac{1}{N^{d-2}}$} \;{\rm cap}(B_N) = \mbox{\f $\dis\frac{1}{d}$} \; {\rm cap}_{\IR^d} ([-1,1]^d)
\end{equation}

\medskip\n
(see E1 on p.~301 of \cite{Spit01}, as well as (2.4) and Lemma 2.1 of \cite{BoltDeus93}). The claim (\ref{3.7}) now follows and Theorem \ref{theo3.2} is proved.
\end{proof}

\begin{remark}\label{rem3.3} ~ \rm

\medskip\n
1) The above simple proof yields a meaningful upper bound only when $\alpha < 0$. It is not clear how the argument can be modified to produce an interesting bound for non-negative values of $\alpha$ below $h_*$ (when $h_* > 0$)..

\bigskip\n
2) The asymptotic upper bound (\ref{3.7}) does not match the asymptotic lower bound of Theorem \ref{theo2.1} when $h_{**} > 0$ (one knows that $0 < h_* \le h_{**}$ for large $d$ and expects this fact to be true for all $d \ge 3$). We will now aim at correcting this defect. \hfill $\square$
\end{remark}

\section{Some Gaussian estimates}
\setcounter{equation}{0}

In this section we develop some bounds on the expectation of the infimum of certain families of Gaussian variables and their variance. Our main result is Theorem \ref{theo4.2}. Its Corollary \ref{cor4.4} will play an important role in the next section. It controls the probability that simultaneously in several boxes the respective harmonic averages of the Gaussian free field attached to these boxes takes values below some fixed negative level. 

\medskip
We first introduce some notation. We consider positive integers
\begin{equation}\label{4.1}
L \ge 1, \;\mbox{and}\; K \ge 100.
\end{equation}

\n
Informally, we are interested in the regime where $L$ tends to infinity and $K$ is large but fixed. Actually, in the next section we will choose $L$ of order $(N \log N)^{\frac{1}{d-1}}$, see (\ref{5.16}), with $N$ having the same interpretation as in (\ref{2.1}), and we will successively let $N$ and $K$ tend to infinity.

\medskip
We introduce the lattice
\begin{equation}\label{4.2}
\IL = L \IZ^d,
\end{equation}
and the boxes in $\IZ^d$
\begin{equation}\label{4.3}
\begin{split}
B_0 = &\; [0,L)^d \subseteq D_0 = [-3L, 4L)^d \subseteq U_0 =  [-KL + 1, L+ K  L-1)^d 
\\
 \subseteq &\; \wt{B}_0 = [-KL, L + KL)^d,
\end{split}
\end{equation}

\medskip\n
as well as their translates to the various sites of $\IL$:
\begin{equation}\label{4.4}
B_z = z + B_0 \subseteq D_z = z + D_0 \subseteq U_z = z + U_0 \subseteq \wt{B}_z = z + \wt{B}_0 .
\end{equation}
We will often refer to the boxes $B_z$, $z \in \IL$, as $L$-boxes. We will use the collection of boxes $U_z, z \in \IL$, to decompose the Gaussian free field. Specifically, as in (\ref{1.17}), with $U = U_z$, $z \in \IL$, we write
\begin{equation}\label{4.5}
\varphi = h^z + \psi^z
\end{equation}

\n
for the corresponding decomposition. Often, for convenience, when $B = B_z$, we will write $h_B$ and $\psi_B$ in place of $h^z$ and $\psi^z$, and refer to $h_B$ as the harmonic average (of $\varphi)$ attached to $B$, and to $\psi_B$ as the local field attached to $B$. Note that for $z \in \IL$,
\begin{align}
&\mbox{$\psi^z$ is independent of $\sigma(\varphi_y, y \in U^c_z)$ \quad (by (\ref{1.18}))}\label{4.6}
\\[1ex]
& \mbox{$\psi^z$ is $\sigma(\varphi_y, y \in \wt{B}_z)$-measurable \qquad \; (by (\ref{1.16}), (\ref{1.15}))}. \label{4.7}
\end{align}

\n
The next lemma collects some independence properties of the above Gaussian random fields and will also be helpful in the next section.
\begin{lemma}\label{lem4.1}
For $z,z'$ in $\IL$ one has 
\begin{equation}\label{4.8}
\wt{B}_z \cap U_{z'} = \phi = U_z \cap \wt{B}_{z'}, \;\mbox{when}\;  |z - z'|_\infty \ge L + 2KL.
\end{equation}
If $\cC \subseteq \IL$ is a collection of sites with mutual $|\cdot |_\infty$-distance at least $L + 2KL$, then
\begin{equation}\label{4.9}
\begin{array}{l}
\mbox{the centered Gaussian fields $\psi^z$, $z \in \cC$, are independent,}
\\
\mbox{and also independent from the collection $(h^z_x)_{x \in \wt{B}^z, z \in \cC}$}.
\end{array}
\end{equation}
\end{lemma}

\begin{proof}
We begin with (\ref{4.8}). The condition on $z,z'$ ensures that $\wt{B}_z \cap \wt{B}_{z'} = \phi$ and (\ref{4.8}) follows, see (\ref{4.3}), (\ref{4.4}). As for the claim (\ref{4.9}), first note that by (\ref{4.8}) and (\ref{4.6}), (\ref{4.7}), the centered Gaussian fields $\psi^z$, $z \in \cC$, are pairwise orthogonal. In addition, each $\psi^z$ is orthogonal to $h^z$ (by (\ref{1.18})), and for $z' \not= z$ in $\cC$, $x' \in \wt{B}_{z'}$, $\psi^z$ is orthogonal to $h^{z'}_{x'}$ (which is $\sigma(\varphi_y, y \in \wt{B}_{z'})$-measurable by (\ref{1.15}), where $\wt{B}_{z'} \subseteq U^c_z$ by (\ref{4.8})). Since all the random fields are centered and jointly Gaussian, the claim (\ref{4.9}) follows.
\end{proof}

\medskip
We now consider (see (\ref{4.2}) for notation)
\begin{equation}\label{4.10}
\begin{array}{l}
\mbox{$\cC$ a non-empty finite subset of $\IL$ with points}
\\
\mbox{at mutual $|\cdot|_\infty$-distance at least $L + 2 KL$}.
\end{array}
\end{equation}
Given $\cC$ as above we write
\begin{equation}\label{4.11}
C = \bigcup\limits_{z \in \cC} B_z.
\end{equation}
As a shorthand, we also write $B \in \cC$ to mean $B = B_z$ with $z \in \cC$. we denote by $\nu$ the equilibrium measure of $C$ and by $\ov{\nu} = \frac{\nu}{{\rm cap}(C)}$ the normalized equilibrium measure of $C$.

\medskip
We attach to $\cC$ the collection $\cF$ of functions $f$ from $\cC$ into $\IZ^d$ such that the image of any $z$ belongs to $D_z$:
\begin{equation}\label{4.12}
\cF = \{f \in (\IZ^d)^\cC; \; f(z) \in D_z \;\mbox{for each $z \in \cC\}$}.
\end{equation}

\n
As mentioned above, we often view $\cC$ as a collection of $L$-boxes and also write $f(B)$ in place of $f(z)$, when $f \in \cF$ and $B = B_z$. Further, we attach to $\cC$ the probability on $\cC$ with weight 
\begin{equation}\label{4.13}
\lambda(z) = \ov{\nu}(B_z) \; \mbox{for each} \; z \in \cC.
\end{equation}

\n
(note that the boxes $B_z$, $z \in \cC$, are pairwise disjoint and $\sum_{z \in \cC} \lambda(z) = 1$). We will routinely write $\lambda(B)$ in place of $\lambda(z)$ when $B = B_z$.

\medskip
The centered Gaussian field (indexed by $\cF$) that we now introduce, plays an important role in this section. Specifically, we set (see below (\ref{4.5}) for notation)
\begin{equation}\label{4.14}
Z_f = \dis\sum\limits_{B \in \cC} \lambda(B) \,h_B\big(f(B)\big),
\end{equation}
as well as
\begin{equation}\label{4.15}
Z = \inf\limits_{f \in \cF} Z_f .
\end{equation}

\medskip\n
We will use $(Z_f)_{f \in \cF}$ and $Z$ as tools in order to bound the probability that $\inf_{D_z} h^z \le - a$, for each $z$ in $\cC$ (with $a$ some positive number). These bounds will rely on uniform controls on the variance of $Z_f$ and on the expectation of $Z$. These controls are encapsulated in the next theorem.

\begin{theorem}\label{theo4.2} (recall that $L \ge 1$ and $K \ge 100$, cf.~(\ref{4.1}))
\begin{equation}\label{4.16}
\limsup\limits_L\;  \sup\limits_\cC \;\sup\limits_{f \in \cF} \{{\rm var}(Z_f)\,{\rm cap}(C)) < \alpha(K), \;\mbox{where $\alpha (K) > 1$ and $\lim\limits_K \alpha (K)= 1$}
\end{equation}
(the supremum over $\cC$ runs over all collections as in (\ref{4.10}), and the notation is the same as in (\ref{4.11}), (\ref{4.12})). Moreover, one has
\begin{equation}\label{4.17}
\sup\limits_\cC \,|\IE [Z]| \;\Big(\mbox{\f $\dis\frac{|\cC|}{{\rm cap}(C)}$}\Big)^{-\frac{1}{2}} \le \mbox{\f $\dis\frac{c_4}{K}$}\,.
\end{equation}
\end{theorem}

\begin{proof}
We first prove (\ref{4.16}). We consider a collection $\cC$ as in (\ref{4.10}). By (\ref{4.8}), when $z \not= z'$ belong to $\cC$, $U_z \cap \wt{B}_{z'} = \phi = U_{z'} \cap \wt{B}_z$. Writing $B$ and $B'$ in place of $B_z$ and $B_{z'}$, $D$ and $D'$ in place of $D_z$ and $D_{z'}$, $U$ and $U'$ in place of $U_z$ and $U_{z'}$ (and so on), we see that when $x \in D$, $x' \in D'$, with $z \not= z'$ in $\cC$, then
\begin{equation}\label{4.18}
\begin{split}
\IE[h_B(x) \,h_B(x')]  \stackrel{(\ref{1.15})}{=} &\; \dis\sum\limits_{y,y'} P_x[X_{T_U} = y] \,P_{x'} [X_{T_{U'}} = y'] \,g(y,y')
\\
\stackrel{(\ref{1.4})}{=} \, &\; \dis\sum\limits_{y'} P_{x'} [X_{T_{U'}} = y'] \,g(x,y') \quad \mbox{(since $\partial U' \subseteq U^c$)}
 \\
 \stackrel{(\ref{1.4})}{=}\, &\; g(x,x') \quad \mbox{(since $x \notin U'$)}.
\end{split}
\end{equation}

\medskip\n
On the other hand, when $z = z' \in \cC$, we find that for $x,x' \in D$
\begin{equation}\label{4.19}
\IE[h_B(x) \,h_B(x')] = \dis\sum\limits_{y'} P_{x'} [X_{T_U} = y'] \,g(x,y') = \dis\sum\limits_y \,P_x [X_{T_U} = y] \,g(y,x')
\end{equation}

\n
(for the first equality one uses (\ref{1.15}), (\ref{1.4}), and that $y'$ belongs to $\partial U \subseteq U^c$, and for the second equality  (\ref{1.15}), (\ref{1.4}), and that $y$ belongs to $\in \partial U \subseteq U^c$).

\medskip
Coming back to the definition of $Z_f$ in (\ref{4.14}) we see that
\begin{equation}\label{4.20}
\begin{split}
{\rm var}(Z_f) & =  \dis\sum\limits_{B,B' \in \cC} \lambda (B) \,\lambda(B') \, \IE\big[h_B\big(f(B)\big) \,h_{B'}\big(f(B')\big)\big]
\\[1ex]
&\!\!\! \stackrel{(\ref{4.18})}{=}  \dis\sum\limits_{B \in \cC} \lambda(B)^2 \IE \big[h^2_B \big(f(B)\big)\big] + \dis\sum\limits_{B \not= B'} \lambda(B) \,\lambda(B') \,g\big(f(B),f(B')\big).
\end{split}
\end{equation}
We then introduce
\begin{equation}\label{4.21}
\gamma(K,L) = \wt{\sup} \,g(y,y') / g(x,x') \; (\ge 1)
\end{equation}

\n
where $\wt{\sup}$ denotes the supremum over $y \in D_z$, $y' \in D_{z'}$, $x \in B_{z'}$ and $x' \in B_{z'}$, with $z,z' \in \IL$ such that $|z - z'|_\infty \ge L + 2KL$. It readily follows from (\ref{1.2}) that
\begin{equation}\label{4.22}
\lim\limits_K \; \limsup\limits_L \; \gamma(K,L) = 1.
\end{equation}

\n
We will now bound the expression in the last line of (\ref{4.20}). To handle the first sum, we use (\ref{4.19}), (\ref{1.2}) and the fact that $d(\partial U, D) \ge (K-3) L$. We thus find that for any $B \in \cC$ and $y \in D$ one has
\begin{equation}\label{4.23}
\IE[h_B (y)^2] \le \dis\frac{c}{(KL)^{d-2}} \,.
\end{equation}

\medskip\n
Further, we note that for any $B \in \cC$,
\begin{equation}\label{4.24}
\lambda(B) \underset{(\ref{1.5})}{\stackrel{(\ref{4.13})}{=}} \;\mbox{\f $\dis\frac{1}{{\rm cap}(C)}$} \;\dis\sum\limits_{x \in B} P_x[\wt{H}_C = \infty] \le \mbox{\f $\dis\frac{1}{{\rm cap}(C)}$}  \;\dis\sum\limits_{x \in B} P_x[\wt{H}_B = \infty] \stackrel{(\ref{1.5})}{=} \mbox{\f $\dis\frac{{\rm cap}(B)}{{\rm cap}(C)}$}.
\end{equation}

\medskip\n
Hence, coming back to (\ref{4.20}), keeping in mind that $\sum_B \lambda(B) = 1$, we see that for all $\cC$ as in (\ref{4.10}) and $f$ in $\cF$
\begin{equation}\label{4.25}
\begin{split}
{\rm var}(Z_f) & \le \mbox{\f $\dis\frac{c}{(KL)^{d-2}}$} \; \mbox{\f $\dis\frac{{\rm cap}(B_0)}{{\rm cap}(C)}$} + \gamma(K,L) \;\dis\sum\limits_{B\not= B'} \; \dis\sum\limits_{x \in B, x'\in B'} \ov{\nu}(x) \,\ov{\nu}(x') \,g(x,x')
\\
&\!\!\! \stackrel{(\ref{1.7})}{\le}  \mbox{\f $\dis\frac{c}{K^{d-2}}$} \; \mbox{\f $\dis\frac{1}{{\rm cap}(C)}$} + \mbox{\f $\dis\frac{\gamma(K,L)}{{\rm cap}(C)^2}$} \; \dis\sum\limits_{x,x' \not= C} \nu(x) \,\nu(x')\,g(x,x')
\\
& \!= \mbox{\f $\dis\frac{1}{{\rm cap}(C)}$} \;\Big(\mbox{\f $\dis\frac{c}{K^{d-2}}$} + \gamma(K,L)\Big) \; \mbox{(since $\nu$ is the equilibrium measure of $C$)}.
\end{split}
\end{equation}

\medskip\n
The claim in (\ref{4.16}) now follow by (\ref{4.22}).

\medskip
We then turn to the proof of (\ref{4.17}). We will use bounds in the expectation of the infimum of the centered Gaussian process $Z_f, f \in \cF$, based on the metric entropy, see \cite{AdleTayl07}, p.~14. For this purpose, we need a control on the regularity of the map $f \in \cF \rightarrow Z_f \in L^2(\IP)$. This is precisely the object of the next lemma.
\begin{lemma}\label{lem4.3} $(L \ge 1, K \ge 100)$

\medskip
For all $\cC$ as in (\ref{4.10}), with $C$ as in (\ref{4.11}), and $\cF$ as in (\ref{4.12}), one has
\begin{equation}\label{4.26}
\IE[(Z_f - Z_k)^2] \,{\rm cap}(C) \le c^2_5 \; \mbox{\f $\dis\frac{\| f-k\|^2_\infty}{(KL)^2}$} \,, \; \mbox{for all $f,k \in \cF$},
\end{equation}

\medskip\n
where $\|f-k\|_\infty$ stands for $\sup_{B \in \cC} |f(B) - k(B)|_\infty$.
\end{lemma}

\begin{proof}
For $\cC$ as above, and $f,k \in \cF$, one has (with hopefully obvious notation)
\begin{equation}\label{4.27}
\begin{array}{l}
\IE[(Z_f - Z_k)^2] = \IE\Big[\Big(\dis\sum\limits_B \lambda(B)(h_B\big(f(B)\big) - h_B\big(k(B)\big)\Big)^2\Big] =
\\[2ex]
\dis\sum\limits_{B,B'} \lambda(B) \,\lambda(B') \,\IE\big[\big(h_B\big(f(B)\big) - h_B\big(k(B)\big) \big(h_{B'}\big(f(B')\big) - h_{B'}\big(k(B')\big)\big)\big] \stackrel{(\ref{1.15})}{=}
\\
\\[-1ex]
\dis\sum\limits_{B,B'} \lambda(B) \,\lambda(B') \dis\sum\limits_{z,z'} (P_{f(B)}[X_{T_U} = z] - P_{k(B)}[X_{T_U} = z])(P_{f(B')}[X_{T_{U'}} = z'] 
\\
\\[-1ex]
- P_{k(B')} [X_{T_{U'}} = z']) \,\IE[\varphi_z \varphi_{z'}].
\end{array}
\end{equation}

\medskip\n
Next, we observe that $P_\point[X_{T_U} = z]$, for $z \in \partial U$, is a non-negative harmonic function in $U$. Combining the Harnack Inequality and the gradient estimates in Theorems 1.7.2 and 1.7.1, p.~42 of \cite{Lawl91}, we see that
\begin{equation}\label{4.28}
\begin{array}{l}
|P_x[X_{T_U} = z] - P_y[X_{T_U} = z] \,| \le 
\\[2ex]
c \;\mbox{\f $\dis\frac{|x-y|_\infty}{KL}$} P_x[X_{T_U} = z], \;\mbox{for any $x,y \in D$ and $z \in \partial U$}.
\end{array}
\end{equation}

\medskip\n
Applying the same bound to $U'$, and keeping in mind that $\IE[\varphi_z \varphi_{z'}] = g(z,z') \ge 0$ in the last line of (\ref{4.27}), we thus find that
\begin{equation}\label{4.29}
\begin{split}
\IE[(Z_f - Z_k)^2] & \le \;c \dis\sum\limits_{B,B'} \lambda(B) \,\lambda(B') \,\mbox{\f $\dis\frac{|f(B) - k(B)|_\infty |f(B')-k(B')|_\infty}{(KL)^2}$} 
\\[1ex]
& \hspace{1cm} \IE\big[h_B\big(f(B)\big)\,h_{B'}\big(f(B')\big)\big]
\\[1ex]
& \le \;c \;\mbox{\f $\dis\frac{\|f-k\|^2_\infty}{(KL)^2}$} \;\IE[Z_f^2] \le\; c \;\mbox{\f $\dis\frac{\|f-k\|^2_\infty}{(KL)^2}$} \; \mbox{\f $\dis\frac{1}{{\rm cap}(C)}$}\,,
\end{split}
\end{equation}

\medskip\n
where in the last inequality we used (\ref{4.25}) and the  fact that $\gamma(K,L)$ in (\ref{4.21}) is smaller than $\sup g(y)/g(x)$, with a supremum over $|x|_\infty \ge 2K L$ and $|y - x|_\infty \le 14 L$, so that $\frac{1}{2} \,|x|_\infty \le |y|_\infty \le 2|x|_\infty$, whence $\gamma(K,L) \le c$. This completes the proof of (\ref{4.26}).
\end{proof}

\medskip
We now resume the proof of (\ref{4.17}). We pick $\cC$ as in (\ref{4.10}) and, for convenience, introduce the scaled centered Gaussian process
\begin{equation}\label{4.30}
\wt{Z}_f = \sqrt{{\rm cap}(C)} \;Z_f, \;\mbox{for $f \in \cF$}.
\end{equation}

\n
By (\ref{4.26}) of the above lemma, the so-called canonical metric on $\cF$ induced by the $L^2(\IP)$-distance on $\wt{Z}_f$, $f \in \cF$, satisfies
\begin{equation}\label{4.31}
\IE[(\wt{Z}_f - \wt{Z}_k)^2]^{\frac{1}{2}} \le \dis\frac{c_5}{KL} \;\|f-k\|_\infty \le \dis\frac{7c_5}{K} , \; \mbox{for $f,k \in \cF$}
\end{equation}

\medskip\n
(we used the definition of $\cF$ in (\ref{4.12}) for the last inequality).

\medskip
Thus, given $\ve \in (0, \frac{7c_5}{K}]$, we can cover $\cF$ by balls of radius $\ve$ in the canonical metric as follows. We pick $\ell$ as the largest integer such that $c_5 \frac{\ell}{KL} \le \ve$, i.e. $\ell = [\frac{KL}{c_5}\;\ve]$ (and hence $1 + \ell \le 8 L$, by the bound on $\ve$). We then partition $D_0$ (and by translation invariance each $D_z$, $z \in \IL$) into disjoint boxes having each $|\cdot|_\infty$-diameter at most $\ell$ (so the projection on each axis of such a box  contains at most $\ell + 1$ points). Such a partition can be achieved with at most $([\frac{7L}{\ell + 1}] + 1)^d \le (\frac{15L}{\ell + 1})^d$ boxes (recall that $\ell + 1 \le 8 L)$. Moreover, if $f,k \in \cF$ are such that $f(z)$ and $k(z)$ belong to the same box in $D_z$, for each $z \in \cC$, then, by (\ref{4.31}) and the choice of $\ell$, the canonical distance between $f$ and $k$ is at most $\ve$. In this fashion, we can cover $\cF$ by at most $(\frac{15L}{\ell + 1})^{d|\cC|}$ balls of radius $\ve$ for the canonical metric on $\cF$. Since $\ell + 1 \ge \frac{KL}{c_5} \,\ve$, we see that for $0 < \ve \le \frac{7 c_5}{K}$,
\begin{equation}\label{4.32}
\begin{array}{l}
\mbox{$\cF$ is covered by at most $N(\ve) = \Big(\mbox{\f $\dis\frac{15 c_5}{K \ve}$}\Big)^{d|\cC|}$ balls or radius $\ve$}
\\
\mbox{in the canonical metric}.
\end{array}
\end{equation}

\medskip\n
By Theorem 1.3.3, p.~14 of \cite{AdleTayl07}, we find that for all $\cC$ as in (\ref{4.10})
\begin{equation}\label{4.33}
\begin{array}{l}
|\sqrt{{\rm cap}(C)} \;\IE[Z]| = | \IE[\inf\limits_\cF \wt{Z}_f]| \le c \dis\int_0^{\frac{7c_5}{K}} \sqrt{\log N(\ve)} \; d\ve =
\\
\dis\int_0^{\frac{7c_5}{K}} \sqrt{d |\cC| \log \Big(\mbox{\f $\dis\frac{15c_5}{K\ve}$}\Big)} \; d \ve =\mbox{\f $\dis\frac{c'}{K}$} \;\sqrt{|\cC|} \; \dis\int^1_0 \;\sqrt{\log \Big(\mbox{\f $\dis\frac{15}{7\eta}$}\Big)} \;d \eta = \mbox{\f $\dis\frac{c}{K}$} \; \sqrt{|\cC|}.
\end{array}
\end{equation}

\medskip\n
This completes the proof of (\ref{4.17}) and hence of Theorem \ref{theo4.2}.
\end{proof}

We can now combine Theorem \ref{theo4.2} and the Borell-TIS Inequality (see \cite{AdleTayl07}, p.~50) to obtain a bound in the probability that $\inf_D h_B \le -a$, for all boxes $B$ belonging to a finite collection $\cC$ as in (\ref{4.10}). This estimate will play an important role in the next section.
\begin{corollary}\label{cor4.4} ($a > 0, K \ge 100$, and $\alpha (K)$ as in (\ref{4.16}))
\begin{equation}\label{4.34}
\begin{array}{l}
\limsup\limits_L \;\sup\limits_\cC \bigg\{\log \IP[\bigcap\limits_{B \in \cC} \{\inf\limits_D h_B \le - a\}] + \fr \bigg(a - \mbox{\f $\dis\frac{c_4}{K}$} \;\sqrt{\mbox{\f $\dis\frac{|\cC|}{{\rm cap}(C)}$}}\bigg)^2_+ \; \mbox{\f $\dis\frac{{\rm cap}(C)}{\alpha(K)}$}\bigg\} \le 0
\\
\\[-2ex]
\mbox{(recall that $\lim\limits_K \alpha(K) = 1$)}.
\end{array}
\end{equation}
\end{corollary}

\begin{proof}
We consider $L \ge 1$, $K \ge 100$, and $\cC$ as in (\ref{4.10}). We observe that on the event $A = \cap_{B \in \cC} \{\inf_D h_B \le -a\}$, we can choose $f \in \cF$ such that $h_B(f(B)) \le -a$ for each $B \in \cC$, and hence, $Z_f = \sum_B \lambda(B) h_B(f(B)) \le -a$ (recall that $\lambda(\cdot)$ is a probability on $\cC$, cf.~(\ref{4.13})). We can thus apply the Borell-TIS Inequality, see Theorem 2.1.1, p.~50 of \cite{AdleTayl07}, and find that
\begin{equation}\label{4.35}
\begin{array}{l}
\IP[A] \le \IP[\inf\limits_\cF  Z_f \le -a] \stackrel{(\ref{4.15})}{=} \IP[Z \le -a ] \le \exp\Big\{- \mbox{\f $\dis\frac{1}{2 \sigma^2}$} \;(a - |E[Z]|)^2_+\Big\}
\\
\mbox{(with $\sigma^2 = \sup\limits_\cF$ var$(Z_f)$)}.
\end{array}
\end{equation}

\n
Taking logarithms and inserting the bound on $|\IE[Z]|$ from (\ref{4.17}), we obtain
\begin{equation}\label{4.36}
\log \IP [A] + \fr \;\bigg(a - \mbox{\f $\dis\frac{c_4}{K}$} \;\sqrt{\mbox{\f $\dis\frac{|\cC|}{{\rm cap}(C)}$}}\bigg)^2_+ \; \dis\mbox{\f $\dis\frac{{\rm cap}(C)}{\sigma^2 {\rm cap}(C)}$} \le 0\,.
\end{equation}

\n
We can further replace $\sigma^2 {\rm cap}(C)$ by $\sup_\cC \sigma^2 {\rm cap}(C)$ in the above inequality and note that by (\ref{4.16}) $\limsup_L \sup_\cC \sigma^2 {\rm cap}(C) < \alpha(K)$. So, taking a supremum over $\cC$ and then letting $L$ tend to infinity we obtain (\ref{4.34}).
\end{proof}

\medskip
We also record for later use an estimate on the tail of $\sup_D |h_B|$, with $B$ an arbitrary $L$-box. In essence, it corresponds to the case $|\cC| = 1$ in the above set-up.
\begin{corollary}\label{cor4.5} ($L \ge 1,K \ge 100, a > 0,B$ an arbitrary $L$-box)
\begin{equation}\label{4.37}
\IP[\sup\limits_D |h_B| \ge a] \le 2 \exp\Big\{- c(KL)^{d-2} \Big(a - \mbox{\f $\dis\frac{c}{KL^{\frac{d-2}{2}}}$}\Big)^2_+\Big\} .
\end{equation}
\end{corollary}

\begin{proof}
We use (\ref{4.17}) when $|\cC| = 1$ to bound $|\IE[\inf_D  h_B]|$ (here, we can actually improve the bound by using throughout in the proof the fact we do not need $\gamma(K,L)$ in the last line of (\ref{4.35}), but we will not need this sharper bound). The same bound holds for $\IE[\sup_D h_B]$. We also know that ${\rm var}(h_B(x)) \le c(KL)^{-(d-2)}$, for $x$ in $D$ by (\ref{4.23}). The claim now follows from the Borell-TIS Inequality.
\end{proof}

\bigskip\bigskip
\section{Disconnection upper bound in the strong percolative regime}
\setcounter{equation}{0}

In this section we derive an asymptotic upper bound on the probability of the disconnection event $A_N$ from (\ref{0.6}) (or (\ref{2.1})) that corresponds to the absence of a path in $E^{\ge \alpha}$ going from $\partial B_N$ to $S_N$. We assume that $\alpha$ is in a ``strongly percolative regime for $E^{\ge \alpha}$'', more precisely, that $\alpha < \ov{h}$, where $\ov{h}$ is the critical value introduced in (\ref{5.3}) below. Although little is known about $\ov{h}$ at present, see Remark \ref{rem5.1}, one may hope that $\ov{h} = h_* = h_{**}$. If the equalities $\ov{h} = h_* = h_{**}$ hold (this is of course an open problem), then the asymptotic upper bound from the main Theorem \ref{theo5.5} of this section and the asymptotic lower bound from Theorem \ref{theo2.1} actually match.

\medskip
It may be useful to provide at this point an informal outline of the strategy of the proof of the main asymptotic upper bound. We choose the scale $L$ from the previous section, of order $(N \log N)^{\frac{1}{d-1}}$ and attach to each $L$-box a decomposition of the Gaussian free field into an harmonic average of the field and a local field, see (\ref{4.5}). We consider the various ``columns'' of $L$-boxes going from the surface of $B_N$ to $S_N$, and show that, up to a super-exponentially decaying probability, in most columns, the local field in each box of the column is in a percolative mode when one looks at levels slightly below $\ov{h}$, see Proposition \ref{prop5.4}. Hence, on the disconnection event $A_N$, up to a super-exponentially decaying probability, there will be in most columns a box where the harmonic average part of the field almost gets below the negative level $-(\ov{h} - \alpha)$ (otherwise the column would offer some path in $E^{\ge \alpha}$ from $\partial B_N$ to $S_N$, due to the good properties of the local fields in all boxes of this column). Bringing into play the upper bounds from Section 4, the derivation of the desired upper bounds on $\IP[A_N]$ will then be reduced to finding uniform lower bounds on ${\rm cap}(C)$, where $C$ corresponds to the union of $L$-boxes that we selected from the columns, see Lemma \ref{lem5.6}.

\medskip
We will now define the critical value $\ov{h}$. We keep the notation of Section 4. Given $\alpha > \beta$ in $\IR$, we say that $\varphi$ strongly percolates at levels $\alpha, \beta$ when (see (\ref{4.3}), (\ref{4.4}) for notation)
\begin{equation}\label{5.1}
\mbox{$\lim\limits_L \; \mbox{\f $\dis\frac{1}{\log L}$} \; \log \IP\Big[B_0 \cap E^{\ge \alpha}$ has no component of diameter at least $\mbox{\f $\dis\frac{L}{10}$}\Big] = -\infty$,}
\end{equation}
and for any $z = Le$, with $|e| = 1$
\begin{equation}\label{5.2}
\begin{split}
\lim\limits_L \; \dis\frac{1}{\log L} \; \log \IP \big[ & \mbox{there exist components of $B_0 \cap E^{\ge \alpha}$ and $B_z \cap E^{\ge \alpha}$}
\\[-2ex]
&\mbox{with diameter at least $\mbox{\f $\dis\frac{L}{10}$}$, which are not connected}
\\[-0.5ex]
& \mbox{in $D \cap E^{\ge \beta}\big] = -\infty$}.
\end{split}
\end{equation}
We then define the critical value
\begin{equation}\label{5.3}
\mbox{$\ov{h} = \sup\{h \in \IR$; for all $h > \alpha > \beta$, $\varphi$ strongly percolates at levels $\alpha,\beta\}$}.
\end{equation}
We will refer to estimates such as in (\ref{5.1}) or (\ref{5.2}), as super-polynomial decay (in $L$) of the probabilities under consideration.

\begin{remark}\label{rem5.1} \rm ~

\medskip\n
1) Using a union bound, it is straightforward to see that when $\varphi$ strongly percolates at levels $\alpha, \beta$, then $\lim_L \IP[B(0,L) \stackrel{\ge \beta}{\longleftrightarrow} \partial B(0,2L)] =1$. In particular, by the definition of $h_{**}$ in (\ref{0.5}), this shows that $\ov{h} \le h_{**}$. Actually, one can patch up crossings in $E^{\ge \beta}$ from $B(0,2^k)$ to $\partial B(0,2^{k+1})$, for $k \ge k_0$, with the help of (\ref{5.2}), and find, with a union bound, that when $\gamma < \beta < \alpha <  \ov{h}$, then $E^{\ge \gamma}$ percolates with positive probability (and hence, probability one, by ergodicity). This shows that
\begin{equation}\label{5.4}
\ov{h}  \le h_{*}  (\le h_{**}).
\end{equation}
2) By similar considerations as in the proof of Theorem 2.5 of \cite{DrewRathSapo14b}, one can also show that
\begin{equation}\label{5.5}
\ov{h} >  - \infty.
\end{equation}
We simply sketch the argument. One considers $\wt{h}_{**}$ defined similarly as $h_{**}$ in (\ref{0.5}), but replacing pathes by $*$-paths in the definition of the event under the probability. One can show a similar estimate as in (\ref{1.19}) when $\alpha > \wt{h}_{**}$, replacing paths by $*$-paths. One can then apply these estimates to the Gaussian free field $-\varphi$ and see that when $\gamma < - \wt{h}_{**}$ the excursion set $E^{< \gamma} (= \{x \in \IZ^d; \varphi_x < \gamma\})$ is ``strongly non $*$-percolative''. From this feature, and the $*$-connectedness of exterior boundaries of finite connected subsets of $\IZ^d$, see for instance Lemma 2 of \cite{Tima12}, one deduces that $\varphi$ strongly percolates at levels $\alpha, \beta$ when $\beta < \alpha < \gamma < - \wt{h}_{**}$. Thus, $\ov{h} \ge - \wt{h}_{**}$ and in particular, the finiteness of $\ov{h}$ claimed in (\ref{5.5}) follows.

\medskip
Little is known about $\ov{h}$ otherwise. It is an open question whether $\ov{h} = h_* = h_{**}$ and even more modestly whether $\ov{h} \ge 0$.

\medskip\n
3) If one chooses $\beta = \alpha$, in place of $\beta < \alpha$, in condition (\ref{5.2}), one requires a more stringent condition (because the probability in this modification of (\ref{5.2}) becomes bigger). The corresponding $\wt{h}$ one obtains in place of $\ov{h}$ as the supremum of the values $h$ such that (\ref{5.1}) and the modified (\ref{5.2}) hold for all $\alpha < h$, satisfies $\wt{h} \le \ov{h}$. In addition, using a union bound, it is routine to see that when (\ref{5.1}) and the modified (\ref{5.2}) hold at level $\alpha$, then $E^{\ge \alpha}$ percolates with positive probability (and hence, with probability one, by ergodicity). One thus finds that $\wt{h} \le h_*$. Looking at the excursion-set $E^{\ge \alpha}$ for $\alpha < \wt{h}$ is very much in the spirit of condition S1 in (1.6) of \cite{DrewRathSapo14b}.

\medskip
The present choice of (\ref{5.1}), (\ref{5.2}) of course leads to an $\ov{h}$ (formally) bigger than $\wt{h}$, but more pragmatically it minimizes the changes we will make in the next section, when working in high dimension.

\medskip\n
4) Let us incidentally point out that in the case of the vacant set of random interlacements at a small positive level $u$, estimates such as (\ref{5.1}) and (\ref{5.2}) (with in fact a much stronger stretched exponential decay, and $E^{\ge \alpha}$, $E^{\ge \beta}$ both replaced by the vacant set $\cV^u$) are known to hold, see \cite{DrewRathSapo14a} and also \cite{Teix11}, when $d \ge 5$. \hfill $\square$
\end{remark}

\medskip
Given $L \ge 1$, $K \ge 100$ and $B$ an $L$-box (see below (\ref{4.4})), we attach to the box $B$ a decomposition of the Gaussian free field as in (\ref{4.5})
\begin{equation}\label{5.6}
\varphi = h_B + \psi_B,
\end{equation}
into an harmonic average of the field and a local field.

\pagebreak
We now introduce a notion of $B$ being $\psi$-good, which will be fulfilled with high probability when the levels involved are smaller than $\ov{h}$ (see Proposition \ref{prop5.2}). More precisely, given $\gamma > \delta$ in $\IR$, we say that an $L$-box $B$ (i.e. $B = B_z$, with $z \in \IL$) is $\psi$-good at levels $\gamma, \delta $ when
\begin{equation}\label{5.7}
\mbox{$B \cap \{x; \psi_B(x) \ge \gamma\}$ contains a component of diameter a least $\mbox{\f $\dis\frac{L}{10}$}$},
\end{equation}
and for any neighboring box $B'$ of $B$ (i.e. $B'= B_{z'}$, with $z' \in \IL$ and $|z' - z| = L$), 
\begin{equation}\label{5.8}
\begin{array}{l}
\mbox{any two components of $B \cap \{x; \psi_B(x) \ge \gamma\}$ and $B' \cap \{x; \psi_{B'}(x) \ge \gamma\}$ with}
\\
\mbox{diameter at least $\frac{L}{10}$ are connected in $D \cap\{x; \psi_B(x) \ge \delta\}$}.
\end{array}
\end{equation}
Otherwise, we say that $B$ is $\psi$-bad at level $\gamma, \delta$. We will see in Proposition \ref{prop5.2} below that when $\ov{h} > \gamma > \delta$, then for large $L$, $L$-boxes are typically $\psi$-good at levels $\gamma, \delta$.

\medskip
We also introduce a notion of $h$-good. Given $a > 0$, we say that an $L$-box $B$ is $h$-good at level $a$ if
\begin{equation}\label{5.9}
\inf\limits_D h_B > -a.
\end{equation}
Otherwise, we say that $B$ is $h$-bad at level $a$ (let us point out that $B$ being $h$-good at level $a$, or $B$ being $h$-bad at level $a$ are $\sigma(\varphi_x,x \in U \cup \partial U)$-measurable events). We already derived estimates on the probability that a collection of boxes satisfying (\ref{4.10}) are $h$-bad in Corollary \ref{cor4.4}. We will now control the probability of occurence of $\psi$-bad boxes when operating at levels below $\ov{h}$.

\begin{proposition}\label{prop5.2} $(K \ge 100)$

\medskip
Assume $\ov{h} > \gamma > \delta$, then (in the notation of (\ref{4.3}))
\begin{equation}\label{5.10}
\lim\limits_L \; \mbox{\f $\dis\frac{1}{\log L}$} \;\log \mbox{$\IP[ B_0$ is $\psi$-bad at levels $\gamma, \delta] = - \infty$}.
\end{equation}
\end{proposition}

\begin{proof}
We pick $\alpha, \beta$ so that $\ov{h} > \alpha > \beta > \gamma > \delta$. We will apply (\ref{5.1}) with $\alpha,\beta$ and (\ref{5.2}) with $\alpha$ replaced by $\frac{\gamma + \delta}{2}$ and $\beta$ by $\frac{\gamma + 3 \delta}{4}$.

\medskip
By (\ref{5.1}) we know that the probability that $B \cap E^{\ge \alpha}$ has no component of diameter at least $\frac{L}{10}$ decays super-polynomially in $L$. By Corollary \ref{cor4.5} we also know that
\begin{equation}\label{5.11}
\lim\limits_L \; \mbox{\f $\dis\frac{1}{\log L}$} \; \log \IP\Big[\sup\limits_D |\varphi -\psi_B| \ge (\alpha - \gamma) \wedge \mbox{\f $\dis\frac{(\gamma - \delta)}{4}$}\Big] = - \infty\,.
\end{equation}
Hence, the probability that $B \cap \{\psi_B \ge \gamma\}$ has no component of diameter at least $\frac{L}{10}$ decays super-polynomially in $L$. This takes care of the probability that (\ref{5.7}) does not hold.

\medskip
We now estimate the probability that (\ref{5.8}) does not hold. When $B$ and $B'$ are neighboring $L$-boxes and $\sup_D |\varphi - \psi_B| \le \frac{1}{4} \,(\gamma - \delta)$ as well as $\sup_{D'} |\varphi - \psi_{B'}| \le \frac{1}{4} \,(\gamma - \delta)$ hold, then the connected sets of $B \cap \{\psi_B \ge \gamma\}$ and $B' \cap \{\psi_{B'} \ge \gamma\}$ with diameter at least $\frac{L}{10}$ are also connected sets of $B \cap E^{\ge \frac{\gamma + \delta}{2}}$ and $B' \cap E^{\ge \frac{\gamma + \delta}{2}}$ respectively, with diameter at least $\frac{L}{10}$. Applying (\ref{5.2}) with $\alpha$ replaced by $\frac{\gamma + \delta}{2}$ and $\beta$ by $\frac{\gamma + 3 \delta}{4}$, up to super-polynomially decreasing probability in $L$, when $\sup_D |\varphi - \psi_B| \le \frac{1}{4} \,(\gamma - \delta)$ and $\sup_{D'} |\varphi - \psi_{B'} | \le \frac{1}{4} \,(\gamma - \delta)$ hold, then any connected components of $B \cap \{\psi_B \ge \gamma\}$ and $B' \cap 
\{\psi_{B'} \ge \gamma\}$ with diameter at least $\frac{L}{10}$ are connected in $D \cap E^{\ge \frac{\gamma + 3 \delta}{4}} \subseteq D \cap \{\psi_B \ge \delta\}$. By (\ref{5.11}) we thus find that the probability that (\ref{5.8}) does not hold decays super-polynomially in $L$, and this completes the proof of (\ref{5.10}).
\end{proof}

\medskip
We will now use paths of good boxes to construct paths in suitable excursion sets of the Gaussian free field, and also state some independence properties of the $\psi$-good boxes. We introduce one further notation:
\begin{equation}\label{5.12}
\ov{K} = 2 K + 3.
\end{equation}
\begin{lemma}\label{lem5.3} (recall $L \ge 1, K\ge 100, \gamma > \delta, a > 0$)

\medskip
If $\cC$ is a subset of $\IL$ of points at mutual distance at least $\ov{K} L$, then
\begin{equation}\label{5.13}
\mbox{the events $B$ is $\psi$-good at levels $\gamma, \delta$, as $B$ runs over $\cC$, are independent}.
\end{equation}
If $B^i$, $0 \le i \le n$, is a sequence of neighboring $L$-boxes, which are $\psi$-good at levels $\gamma, \delta$ and $h$-good at level $a$, then
\begin{equation}\label{5.14}
\begin{array}{l}
\mbox{there exists a path in $E^{\ge \delta - a} \cap \Big(\bigcup\limits^n_{i=0} D^i\Big)$ starting in $B^0$ and ending in $B^n$,}
\\ 
\mbox{(where $D^i$ stands for the $D$-type box attached to $B^i$, see (\ref{4.3}), (\ref{4.4}))}.
\end{array}
\end{equation}
\end{lemma}

\begin{proof}
We first prove (\ref{5.13}). We note that the event $B$ is $\psi$-good at levels $\gamma, \delta$ is measurable with respect to the $\sigma$-algebra generated by the random fields $\psi_{B'}$, where $B'$ is either $B$ or a neighbor of $B$. As $B$ runs over $\cC$, these collections (viewed as subsets of $\IL$) are at mutual distance at least $L + 2 KL$. By (\ref{4.9}) one has pairwise orthogonality of the centered Gaussian processes $\psi_{B'}$ (with $B'$ equal to $B$ or a neighbor of $B$), as $B$ runs over $\cC$. Since these random fields are jointly Gaussian, the claim (\ref{5.13}) follows.

\medskip
As for (\ref{5.14}), we observe that when $B$ and $B'$ are neighbors and both $\psi$-good at levels $\gamma, \delta$, then there exists components of diameter at least $\frac{L}{10}$ in $B \cap \{\psi_B \ge \gamma\}$ and $B' \cap \{\psi_{B'} \ge \gamma\}$ and any such components are connected in $D \cap \{\psi_B \ge \delta\}$. If $B$ and $B'$ are $h$-good at level $a$ as well, then these components lie in, and are connected in $D \cap E^{\ge \delta - a}$. Using induction, the claim (\ref{5.14}) follows. 
\end{proof}

\medskip
We will now specify the level $\alpha$ entering the definition of $A_N$ in (\ref{0.6}) or (\ref{2.1}), and the relation between the scales $L$ and $N$. From now on, we fix $\alpha$ in $\IR$ such that
\begin{equation}\label{5.15}
\alpha < \ov{h}.
\end{equation}
We pick a (large) $\Gamma \ge 1$ and set
\begin{equation}\label{5.16}
L = [(\Gamma N \log N)^{\frac{1}{d-1}}] \quad \mbox{(we recall that $\IL = L \IZ^d$, see (\ref{4.2}))}.
\end{equation}
We then choose the parameters $\gamma > \delta$ and $a > 0$ respectively entering the notions of $\psi$-good boxes and $h$-good boxes so that
\begin{equation}\label{5.17}
\alpha + a = \delta < \gamma < \ov{h}
\end{equation}
(we will eventually let $a$ tend to $\ov{h} - \alpha$, and $\gamma, \delta$ to $\ov{h}$).

\medskip
We then introduce (with $B$ an arbitrary $L$-box, see below (\ref{4.4}))
\begin{equation}\label{5.18}
\mbox{$\eta = \IP[B$ is $\psi$-bad at levels $\gamma, \delta]$ and $\rho = \sqrt{\log L / \log (\frac{1}{\eta})} \; \underset{L}{\stackrel{(\ref{5.10})}{\longrightarrow}} 0$}.
\end{equation}

\medskip
As a preparation for the proof of the main upper bound on $\IP[A_N]$, we will first see that ``typically, there are few columns of $\psi$-bad boxes''. We first need some notation.

\medskip
Given $e \in \IZ^d$, with $|e| = 1$ and $N \ge 1$, we denote by $F_{e,N}$ the face in the direction $e$ of $B_N$, namely the set $\{x \in B_N$; $x \cdot e = N\}$. For each face $F_{e,N}$, we consider the set of columns, where a column consists of $L$-boxes contained in $\{x \in \IZ^d; x \cdot e > N\} \cap B_{(M+1)N}$ (recall $M > 1$ from above (\ref{2.1})), with same projection in the $e$-direction on $\{x \in \IZ^d$; $x \cdot e = N\}$, which we require to be contained in the face $F_{e,N}$ of $B_N$.

\medskip
As mentioned above, we will now see that few columns contain a $\psi$-bad box. To this end, we introduce the event
\begin{equation}\label{5.19}
\mbox{$C_N = \{$there are at least $\rho(\frac{N}{L})^{d-1}$ columns containing some $\psi$-bad box$\}$}
\end{equation}
with $\rho$ as in (\ref{5.18}) and $\psi$-bad corresponding to the levels $\gamma, \delta$ from (\ref{5.17}). We will now show the super-exponential decay (at speed $N^{d-2}$) of the probability of $C_N$.
\begin{proposition}\label{prop5.4}
\begin{equation}\label{5.20}
\lim\limits_N \; \dis\frac{1}{N^{d-2}} \;\log \IP[C_N] = - \infty .
\end{equation}
\end{proposition}

\begin{proof}
We denote by $m$ the total number of $L$-boxes, which belong to the union of columns. Then, for large $N$,
\begin{equation}\label{5.21}
c_6 \Big(\mbox{\f $\dis\frac{N}{L}$}\Big)^d \le m \le c_7 \Big(\mbox{\f $\dis\frac{MN}{L}$}\Big)^d.
\end{equation}

\n
We write $\IL$ as the disjoint union of the sets $y + \ov{K} \IL$ (recall from (\ref{5.12}) that $\ov{K} = 2 K + 3$), where $y$ varies over $\{0,L,2L,\dots,(\ov{K}-1)L\}^d$. By (\ref{5.13}) in Lemma \ref{lem5.3}, we know that for fixed $y$, the events $\{B$ is $\psi$-bad$\}$, as $B$ varies over $B_z$, with $z \in y + \ov{K} \IL$, are independent. One can thus partition the collection of boxes in the union of columns into $\ov{K}^d$ sub-collections. On the event $C_N$ at least one of these sub-collections contains $\frac{\rho}{\ov{K}^d} \,(\frac{N}{L})^{d-1}$ $\psi$-bad boxes. Of course, each sub-collection contains at most $m$ boxes. We are going to use standard bounds on the probability that a sum of $m$ independent Bernoulli variables with success probability $\eta$ as in (\ref{5.18}) exceed $\frac{\rho}{\ov{K}^d} \,(\frac{N}{L})^{d-1}$. To this end we introduce
\begin{equation}\label{5.22}
\wt{\rho} = \mbox{\f $\dis\frac{\rho}{\ov{K}\,{\!^d}}$} \;\Big(\mbox{\f $\dis\frac{N}{L}$}\Big)^{d-1} \; \mbox{\f $\dis\frac{1}{m}$} \in \Big(\mbox{\f $\dis\frac{\rho}{c_7(\ov{K}M)^d}$} \; \mbox{\f $\dis\frac{L}{N}$} , \; \mbox{\f $\dis\frac{\rho}{c_6 \ov{K}\,{\!^d}}$} \mbox{ \f $\dis\frac{L}{N}$}\Big) \quad \mbox{(by (\ref{5.21}))}.
\end{equation}
As we now explain,
\begin{equation}\label{5.23}
\eta = o(\wt{\rho}), \;\mbox{and $\wt{\rho} \rightarrow 0$, as $N \r \infty$}.
\end{equation}
Indeed, by (\ref{5.18}) and (\ref{5.22}), $\lim_N \,\wt{\rho} = 0$ and
\begin{equation}\label{5.24}
\begin{array}{lcl}
\log \mbox{\f $\dis\frac{\wt{\rho}}{\eta}$} &\!\!\! = &\!\!\!  \log \rho + \log \mbox{\f $\dis\frac{1}{\eta}$} + 
\log \Big( \mbox{\f $\dis\frac{1}{\ov{K}\,{\!^d} m}$} \;\Big(\mbox{\f $\dis\frac{N}{L}$}\Big)^{d-1}\Big)
\\
&\!\!\!  \stackrel{(\ref{5.18})}{=} &\!\!\!  \fr   \log \log L - \fr \;\log \log \mbox{\f $\dis\frac{1}{\eta}$} + \log \mbox{\f $\dis\frac{1}{\eta}$} + \log \Big(\mbox{\f $\dis\frac{1}{\ov{K}\,{\!^d}m}$} \;\Big(\mbox{\f $\dis\frac{N}{L}$}\Big)^{d-1}\Big)
\\
& \!\!\! \sim &\!\!\!  \log \mbox{\f $\dis\frac{1}{\eta}$}  \; (\r \infty), \;\mbox{as $N \r \infty$},
\end{array}
\end{equation}
because $\log \frac{1}{\eta} / \log L \underset{N}{\r} \infty$ by (\ref{5.18}), and the last term of the second line of (\ref{5.24}) remains $O(\log L)$, as $N \r \infty$, by (\ref{5.21}) and (\ref{5.16}). The claim (\ref{5.23}) is now proven.

\bigskip
We can now use standard exponential bounds on sums of independent Bernoulli variables, and, by the argument explained above (\ref{5.22}), we find that for large $N$
\begin{equation}\label{5.25}
\left\{ \begin{array}{l}
\IP[C_N] \le \ov{K}\,{\!^d} \exp\{- m I_N\}, \;\mbox{where}
\\[1ex]
I_N = \wt{\rho} \log \mbox{\f $\dis\frac{\wt{\rho}}{\eta}$} + (1 - \wt{\rho}) \log \mbox{\f $\dis\frac{1- \wt{\rho}}{1 - \eta}$} .
\end{array}\right.
\end{equation}
We will now see that
\begin{equation}\label{5.26}
N^{d-2} = o(m I_N), \; \mbox{as $N \r \infty$},
\end{equation}

\medskip\n
and our claim (\ref{5.20}) will then follow from (\ref{5.25}). By (\ref{5.23}) and (\ref{5.24}) we see that when $N$ tends to infinity,
\begin{equation*}
I_N \sim \wt{\rho} \log \mbox{\f $\dis\frac{\wt{\rho}}{\eta}$} \sim \wt{\rho} \log \mbox{\f $\dis\frac{1}{\eta}$} .
\end{equation*}

\medskip\n
As a result, we find that when $N$ tends to infinity
\begin{equation}\label{5.27}
\begin{split}
mI_N & \sim \mbox{\f $\dis\frac{\rho}{\ov{K}\,{\!^d}}$} \Big(\mbox{\f $\dis\frac{N}{L}$}\Big)^{d-1}  \log  \mbox{\f $\dis\frac{1}{\eta}$}  \stackrel{(\ref{5.18})}{=} \mbox{\f $\dis\frac{1}{\ov{K}\,{\!^d}}$} \;\sqrt{\log L \log {\textstyle \frac{1}{\eta}}} \; \Big(\mbox{\f $\dis\frac{N}{L}$}\Big)^{d-1} 
\\[2ex]
&\!\!\! \stackrel{(\ref{5.16})}{\sim} \; \mbox{\f $\dis\frac{1}{\Gamma \ov{K}\,{\!^d}}$} \; \mbox{\f $\dis\frac{N^{d-2}}{\log N}$} \;\sqrt{\log L \log {\textstyle \frac{1}{\eta}}}\;.
\end{split}
\end{equation}

\medskip\n
The claim (\ref{5.26}) now follows by (\ref{5.16}) and (\ref{5.18}). As explained above, this completes the proof of (\ref{5.20}).
\end{proof}

We are now ready to state and prove the main result of this section.
\begin{theorem}\label{theo5.5} (recall $\alpha < \ov{h}$, see (\ref{2.1}) and (\ref{5.3}) for notation)
\begin{equation}\label{5.28}
\limsup\limits_N \; \mbox{\f $\dis\frac{1}{N^{d-2}}$} \;\log \IP[A_N] \le - \mbox{\f $\dis\frac{1}{2d}$} \;(\ov{h} - \alpha)^2 {\rm cap}_{\IR^d} ([-1,1]^d).
\end{equation}
(of course, when $\alpha < \wt{h}$, with $\wt{h} \le \ov{h}$ defined in Remark \ref{rem5.1} 3), (\ref{5.28}) holds with $\wt{h}$ in place of $\ov{h}$).
\end{theorem}

\smallskip
\begin{proof}
We first assume that
\begin{equation}\label{5.29}
M \ge 2.
\end{equation}
We pick $0 < \kappa < \frac{1}{10}$, and introduce the event (recall $S_N = \{x \in \IZ^d; |x|_\infty = [MN]\})$
\begin{equation}\label{5.30}
\wh{A}_N = \{B_{(1 + \kappa)N}\; \overset{^{\mbox{\footnotesize $\ge \alpha$}}}{\mbox{\Large $\longleftrightarrow$}} \hspace{-3.8ex}/ \quad \;S_N\}.
\end{equation}

\medskip\n
As a (main) step in the proof of (\ref{5.28}), we will first show that
\begin{equation}\label{5.31}
\limsup\limits_N \; \mbox{\f $\dis\frac{1}{N^{d-2}}$} \;\log \IP[\wh{A}_N] \le -  \mbox{\f $\dis\frac{1}{2d}$} \;(\ov{h} - \alpha)^2 {\rm cap}_{\IR^d} ([-1,1]^d).
\end{equation}

\medskip\n
The claim (\ref{5.28}) will then quickly follow.

\medskip
We thus begin with the proof of (\ref{5.31}). We know by (\ref{5.14}) of Lemma \ref{lem5.3} that for large $N$, if all boxes in a column are good (i.e. both $\psi$-good and $h$-good, with the parameters chosen in (\ref{5.17})), then there exists a path in $E^{\ge \delta - a} \stackrel{(\ref{5.17})}{=} E^{\ge \alpha}$ from $B_{(1 + \kappa)N}$ to $S_N$. Hence, for large $N$, on $\wh{A}_N$ all columns contain a bad box. Further, by definition of $C_N$ in (\ref{5.19}), we see that for large $N$,
\begin{equation}\label{5.32}
\begin{array}{l}
\mbox{on $D_N = \wh{A}_N \backslash C_N$, except for at most $\rho(\frac{N}{L})^{d-1}$ columns, all columns contain}
\\
\mbox{an $h$-bad box (at level $a$ from (\ref{5.17}))}.
\end{array}
\end{equation}
Thus, on $D_N$ we can select a set of $[\rho (\frac{N}{L})^{d-1}]$ columns, and then, for each column in the remaining set of columns, select a box $B$ in the column, which is $h$-bad. We further remove columns that have their projection on the face $F_{e,N}$ attached to the column, at a distance less than $\ov{K} L$ from any other face $F_{e',N}$, $e' \not= e$. Then, restricting to a sub-lattice, we only keep columns attached to each given face $F_{e,N}$, with $|e| = 1$, which are at mutual $|\cdot |_\infty$-distance at least $\ov{K} L$ (specifically, we only keep columns of boxes that have labels $z \in \IL$, such that the projection of $z$ on the orthogonal space to $e$ belongs to $\ov{K} \IL$).

\medskip
We write $\wt{C}$ for the subset of $\partial_i B_N = \bigcup_{|e| = 1} F_{e,N}$ (the internal boundary of $B_N)$ obtained by projecting the selected columns onto the face $F_{e,N}$ attached to the respective columns, and we write $F_N$ for the set of points of $\partial_i B_N$ that belong to a single face, and are at $|\cdot |_\infty$-distance at least $\ov{K} L$ from all other faces.

\medskip 
In the fashion described above, we see that there is a family with cardinality at most $\exp\{c_8 (\frac{N}{L})^{d-1} \log ((M+1)N)\}$ of finite subsets $\cC$ of $L$-boxes such that
\begin{equation}\label{5.33}
\begin{array}{l}
\left\{ \begin{array}{rl}
{\rm i)} & \mbox{the boxes $B$ in $\cC$ belong to mutually distinct columns},
\\[1ex]
{\rm ii)} & \mbox{the columns containing a box of $\cC$ are at mutual $|\cdot |_\infty$-distance} \\
&\mbox{at least $\ov{K} L$},
\\[1ex]
{\rm iii)} & \wt{C} \subseteq F_N,
\\[1ex]
{\rm iv)} & \mbox{at most $c_9(\ov{K})(N^{d-2}  L + \rho N^{d-1})$ points of $F_N$ are at $|\cdot|_\infty$-distance}\\
&\mbox{bigger than $\ov{K} L$ of $\wt{C}$};
\end{array}\right.
\end{array}
\end{equation}

\n
(recall $\wt{C}$ is obtained by projecting the boxes of $\cC$ on the face of $B_N$ attached to the column where the box sits).

\medskip
This yield a ``coarse graining'' of the event $D_N$, in the sense that for large $N$
\begin{equation}\label{5.34}
D_N \subseteq \bigcup\limits_\cC D_{N,\cC}, \;\mbox{where $D_{N,\cC} = \bigcap\limits_{B \in \cC} \{B$ is $h$-bad$\}$}
\end{equation}
(and $\cC$ runs over a family with cardinality at most $\exp\{ c_8 (\frac{N}{L})^{d-1}\log ((M+1)N)\}$, with (\ref{5.33}) fulfilled).

\medskip
By (\ref{5.33}) i) and ii), all $\cC$ in the above family satisfy (\ref{4.10}). We can therefore apply Corollary \ref{cor4.4} to bound $\IP[D_{N,\cC}]$ uniformly over $\cC$. We then find that
\begin{equation}\label{5.35}
\begin{array}{l}
\limsup\limits_N \, \mbox{\f $\dis\frac{1}{N^{d-2}}$} \log  \IP[D_N] \le
\\
 \limsup\limits_N \; \Big\{c_8 \Big(\mbox{\f $\dis\frac{N}{L}$}\Big)^{d-1}   \mbox{\f $\dis\frac{\log ((M+1)N)}{N^{d-2}}$}  +  \sup\limits_{\cC}    \mbox{\f $\dis\frac{1}{N^{d-2}}$} \log \IP[D_{N,\cC}]\Big\} \underset{(\ref{4.34})}{\stackrel{(\ref{5.16})}{\le}}
\\[2ex]
\mbox{\f $\dis\frac{c_8}{\Gamma}$} - \liminf\limits_N \, \inf\limits_\cC \Big\{ \mbox{\f $\dis\frac{1}{N^{d-2}}$} \;\fr \,\Big(a - \mbox{\f $\dis\frac{c_4}{K}$} \;\sqrt{\mbox{\f $\dis\frac{|\cC|}{{\rm cap}(C)}$}}\Big)^2_+ \; \mbox{\f $\dis\frac{{\rm cap}(C)}{\alpha(K)}$}\Big\},
\end{array}
\end{equation}

\medskip\n
where we recall that $C = \bigcup_{B \in \cC} B$, cf.~(\ref{4.11}), and $\lim_{K \r \infty} \alpha(K) = 1$, cf.~(\ref{4.16}). Note that by construction $|\cC|$ cannot exceed the total number of columns, and hence, 
$|\cC| \le c (\frac{N}{L})^{d-1} \stackrel{(\ref{5.16})}{\le} \; \frac{c'}{\Gamma} \; \frac{N^{d-2}}{\log N}$ for all $\cC$ in the family.

\medskip
Our next task is to derive an asymptotic lower bound on ${\rm cap}(C)$, as $\cC$ varies over the above family. We clearly lower the capacity of $C$ if, for each $B$ in $\cC$, we replace $B$ by its $(d-1)$-dimensional face closest (and parallel) to the face $F_{e,N}$ of $B_N$ attached to the column where $B$ sits. We denote by $C'$ this new set. If we project each of the $(d-1)$-dimensional faces entering $C'$ on the corresponding face $F_{e,N}$, we recover $\wt{C}$. In doing so, we decrease the Euclidean distance between points sitting in different faces of $C'$, and have the relative position of points within the same face of $C'$ remain unchanged. Moreover, their mutual distance remain at least $\ov{K} L$, if they belong to different faces of $C'$, cf.~(\ref{5.33}) ii) and iii). As a result of the variational characterization (\ref{1.10}) of the capacity and the behavior at infinity of $g(\cdot)$, see (\ref{1.2}), we find that
\begin{equation}\label{5.36}
\begin{array}{l}
\liminf\limits_N \;\inf\limits_{\cC} \;\mbox{\f $\dis\frac{{\rm cap}(C)}{{\rm cap}(\wt{C})}$} \ge \liminf\limits_N \; \inf\limits_{\cC} \; \mbox{\f $\dis\frac{{\rm cap}(C')}{{\rm cap}(\wt{C})}$} \stackrel{(\ref{1.10})}{=} \liminf\limits_N \; \inf\limits_{\cC} \;\mbox{\f $\dis\frac{\inf\limits_{\wt{\nu}} E(\wt{\nu})}{\inf\limits_{\nu '} E(\nu ')}$} 
\\[3ex]
\mbox{(recall $E(\mu) = \dsl_{x,y} \mu(x) \,\mu(y) \,g(x-y)$)},
\end{array}
\end{equation}

\n
where $\wt{\nu}$ and $\nu'$ respectively vary over the set of probabilities supported on $\wt{C}$ and $C'$. We obtain a lower bound in the last expression of (\ref{5.36}) by choosing $\nu'$, depending on $\wt{\nu}$, so that $\wt{\nu}$ is the image of $\nu'$ under the one-to-one ``projection'' transforming $C'$ into $\wt{C}$. From the remarks made above (\ref{5.36}) we see that the last expression of (\ref{5.36}) is at least $1$ and we have shown that
\begin{equation}\label{5.37}
\liminf\limits_N \; \inf\limits_\cC \;\mbox{\f $\dis\frac{{\rm cap}(C)}{{\rm cap}(\wt{C})}$} \ge 1\,.
\end{equation}
We will now derive an asymptotic lower bound on $\inf_\cC {\rm cap}(\wt{C})$.
\begin{lemma}\label{lem5.6}
\begin{equation}\label{5.38}
\liminf\limits_N \; \inf\limits_\cC \;\mbox{\f $\dis\frac{{\rm cap}(\wt{C})}{{\rm cap}(B_N)}$} \ge 1\,.
\end{equation}
\end{lemma}

\begin{proof}
For any $\cC$ in the family entering the infimum, we have $\wt{C} \subseteq F_N \subseteq B_N$, and by the sweeping identity (\ref{1.12}) and (\ref{1.6}) we find that
\begin{equation}\label{5.39}
{\rm cap}(\wt{C}) = P_{e_{B_N}} [H_{\wt{C}} < \infty]\,.
\end{equation}

\n
In addition, $e_{B_N}$ is supported on $\partial_i B_N$ ($= \{x \in \IZ^d$, $|x|_\infty = N\}$). Thus, our claim (\ref{5.38}) will follow once we show that
\begin{equation}\label{5.40}
\lim\limits_N \; \inf\limits_\cC \; \inf\limits_{x \in \partial_i B_N} \;P_x[H_{\wt{C}} < \infty] = 1\,.
\end{equation}

\n
We will prove (\ref{5.40}) by means of a Wiener-type criterion (see for instance \cite{Lawl91}, p.~55), making use of the fact that the sets $\wt{C}$ under consideration are quite sizeable. More precisely, given any $|e| = 1$ and $x \in F_{e,N}$, we consider the successive exit times of simple random walk from $B(x,2^k)$, where $k_0 \le k \le k_1$, with $k_0$ the smallest integer such that
\begin{equation}\label{5.41}
2^{k_0(d-1)} \ge 10\big(c_9(\rho N^{d-1} + N^{d-2} L) + (2d-2) (2N+1)^{d-2}\, \ov{K} L\big)
\end{equation}
and $k_1$ the largest integer such that
\begin{equation}\label{5.42}
2^{k_1} \le \mbox{\f $\dis\frac{N}{2}$}\,.
\end{equation}
By (\ref{5.16}), (\ref{5.18}) we have $\lim_N k_1 - k_0 = \infty$, and $\lim_N \frac{2^{k_0}}{L} = \infty$, by looking at the last expression of (\ref{5.41}). As we now explain, for large $N$, $x \in F_{e,N}$, and $k_0 < k \le k_1$, $\wt{C} \cap B(x,2^k) \cap F_{e,N}$ covers a non-degenerate fraction of $B(x,2^k) \cap F_{e,N}$. Indeed, by (\ref{5.33}) iii) and iv), $\wt{C} \cap F_{e,N}$ consists of disjoint $(d-1)$-dimensional boxes of side-length $L$, contained in $F_N \cap F_{e,N}$, and at most $c_9(\rho N^{d-1} + N^{d-2} L)$ points of $F_N \cap F_{e,N}$ are at distance bigger than $\ov{K} L$ from $\wt{C} \cap F_{e,N}$. We thus find that for large $N$, the number of points of $B(x,2^{k-1}) \cap F_N \cap F_{e,N}$ at distance at most $\ov{K} L$ from $\wt{C}$ is at least
\begin{equation*}
|B(x,2^{k-1}) \cap F_{e,N}| - | F_{e,N} \backslash F_N | - c_9(\rho N^{d-1} + N^{d-2} L) \stackrel{(\ref{5.41}), (\ref{5.42})}{\ge} \mbox{\f $\dis\frac{9}{10}$} \;|B(x,2^{k-1}) \cap F_{e,N}|.
\end{equation*}
Therefore, for large $N$, any $x \in F_{e,N}$, and $k_0 < k \le k_1$, the number of points of $B(x,2^k) \cap F_N \cap F_{e,N}$ at distance at most $\ov{K} L$ from boxes of $\wt{C}$ contained in $B(x,2^k) \cap F_N \cap F_{e,N}$ is at least $c\,|B(x,2^k) \cap F_{e,N}|$, and as a result $|\wt{C} \cap B(x,2^k) \cap F_{e,N}| \ge c(\ov{K}) |B(x,2^k) \cap F_{e,N}|$.

\medskip
The variational characterization (\ref{1.10}) of the capacity (choosing $\nu$ the normalized counting measure on $\wt{C} \cap B(x,2^k) \cap F_{e,N})$ then shows that for large $N$, for any $x \in F_{e,N}$, and any $k_0 < k \le k_1$, ${\rm cap}(\wt{C} \cap B(x,2^k) \cap F_{e,N}) \ge c(\ov{K}) 2^{k(d-2)}$. With such a capacity estimate, it follows that the walk starting on $\partial B(x,2^k)$ has a probability uniformly bounded from below of entering $\wt{C} \cap B(x,2^k) \cap F_{e,N}$ before exiting $B(x,2^{k+1})$, for $k_0 < k \le k_1$. Since $\lim_N k_1 - k_0 = \infty$, the application of the strong Markov property at the successive exit times of $B(x,2^k)$, with $k_0 < k \le k_1$, readily implies (\ref{5.40}) (this is the Wiener-type criterion alluded to above). This completes the proof of Lemma \ref{lem5.6}.
\end{proof}

Hence, by (\ref{5.37}) and the above lemma, we conclude that
\begin{equation}\label{5.43}
\liminf\limits_N \; \inf\limits_\cC \; \mbox{\f $\dis\frac{{\rm cap}(C)}{{\rm cap}(B_N)}$} \ge 1.
\end{equation}

\medskip\n
By the lower bound on ${\rm cap}(B_N)$ in (\ref{1.7}), and the upper bound on $|\cC|$ below (\ref{5.35}), we see that $\lim_N \sup_\cC \frac{|\cC|}{{\rm cap}(C)} = 0$. Thus, coming back to the last line of (\ref{5.35}) we find that
\begin{equation}\label{5.44}
\begin{array}{l}
\limsup\limits_N \;\mbox{\f $\dis\frac{1}{N^{d-2}}$} \;\log \IP[D_N] \le 
\\
\mbox{\f $\dis\frac{c_8}{\Gamma}$}  - \liminf\limits_N \;\fr \;a^2 \; \mbox{\f $\dis\frac{{\rm cap}(B_N)}{\alpha(K) N^{d-2}}$} \stackrel{(\ref{3.14})}{=} \mbox{\f $\dis\frac{c_8}{\Gamma}$}  - \mbox{\f $\dis\frac{1}{2d}$} \;\mbox{\f $\dis\frac{a^2}{\alpha(K)}$}  \;{\rm cap}_{\IR^d} ([-1,1]^d).
\end{array}
\end{equation}

\medskip\n
We can now let $K$ tend to infinity, $\Gamma$ tend to infinity, and $a$ tend to $\ov{h} - \alpha$ (see (\ref{5.17})), and we obtain (\ref{5.31}).

\medskip
We will now deduce (\ref{5.28}) from (\ref{5.31}), and conclude the proof of Theorem \ref{theo5.5}. We introduce $N' = [\frac{N}{1 + \kappa}]$ (with $\kappa$ as above (\ref{5.30})) and $M' = 2(1 + \kappa) M$. We denote by $A'_N$ the event in (\ref{5.30}), where $N$ is replaced by $N'$ and $M$ by $M'$ (note that $M' \ge 2$ fulfills (\ref{5.29})). Then, for large $N$, one has $(1 + \kappa) N' \le N$ and $M' N' \ge MN$ so that $A_N \subseteq A'_N$. By this inclusion and (\ref{5.31}) we find that
\begin{equation}\label{5.45}
\begin{split}
\limsup\limits_N \;\mbox{\f $\dis\frac{1}{N^{d-2}}$} \;\log \IP[A_N] \le &   \; \limsup\limits_N \;\Big(\mbox{\f $\dis\frac{N'}{N}$}\Big)^{d-2} \;\mbox{\f $\dis\frac{1}{N'^{(d-2)}}$} \;\log \IP[A'_N]
\\
\le & -\mbox{\f $\dis\frac{1}{(1 + \kappa)^{d-2}}$}  \;  \mbox{\f $\dis\frac{(\ov{h} - \alpha)^2}{2d}$}   \;{\rm cap}_{\IR^d} ([-1,1]^d).
\end{split}
\end{equation}

\n
Letting $\kappa$ tend to zero, we obtain (\ref{5.28}). This concludes the proof of Theorem \ref{theo5.5}.
\end{proof}

\begin{remark}\label{rem5.7} \rm  Of course, if the equality $\ov{h}  = h_{**}$ holds, then the asymptotic upper bound from Theorem \ref{theo5.5} actually matches the asymptotic lower bound from Theorem \ref{theo2.1}. However, in the present state of knowledge Theorem \ref{theo5.5} suffers from the defect that very little is known about $\ov{h}$ (not even that $\ov{h} \ge 0$). In the next section we will present a version of Theorem \ref{theo5.5} in the case of high dimension, which aims at correcting this defect.  

\hfill $\square$
\end{remark}

\bigskip\bigskip
\section{A disconnection upper bound in high dimension}
\setcounter{equation}{0}

In this section, we derive an asymptotic upper bound on the  disconnection event $A_N$, see (\ref{0.6}) or (\ref{2.1}), which holds in sufficiently high dimension. In this regime, the upper bound we obtain in this section improves on the upper bound of Theorem \ref{theo3.2} (based on a contour argument). We introduce a substitute for the strongly percolative regime corresponding to (\ref{5.1})-(\ref{5.3}), see (\ref{6.1}), (\ref{6.2}) below. Our main results are Theorem \ref{theo6.1} that states conditions under which we can ensure that (\ref{6.1}) and (\ref{6.2}) hold, and Theorem \ref{theo6.2} that contains the main upper bound. The proof of Theorem \ref{theo6.1} uses methods of Section 3 of \cite{RodrSzni13} (where it is shown that $h_* > 0$ in high dimension), and the static renormalization for Bernoulli percolation, see chapter 7 \S 4 of \cite{Grim99}. The proof of Theorem \ref{theo6.2} is, in essence, an adaptation of the proof of the main Theorem \ref{theo5.5} in the last section.

\medskip
We will now introduce the conditions replacing (\ref{5.1}) and (\ref{5.2}). We first need some notation. For each unit vector $e$ in $\IZ^d$ (i.e. $|e| = 1$), we choose mutually orthogonal unit vectors $\wt{e}_i$,  $1 \le i \le 3$, in $\IZ^d$ such that $\wt{e}_1 = e$. We write $H_e$ for the three-dimensional subspace $\IZ\, \wt{e}_1 + \IZ \,\wt{e}_2 + \IZ \,\wt{e}_3$ of $\IZ^d$, and $S_e$ for the two-dimensional slab consisting of points in $H_e$ with $\wt{e}_3$-coordinate in $[0,2 L_0)$.

\medskip
We say that $\alpha > \beta$ and $L_0 \ge 100$ are admissible when for some $e$ (and hence for all $|e| = 1$) one has in the notation of (\ref{4.3}), (\ref{4.4})
\begin{align}
\lim\limits_L  \mbox{\f $\dis\frac{1}{\log L}$}  \log \IP\big[&B_0 \cap S_e \cap E^{\ge \alpha} \;\mbox{has no component of diameter at least}\; \textstyle{\frac{L}{10}}\big] = -\infty,\label{6.1}
\\[1ex]
\lim\limits_L  \mbox{\f $\dis\frac{1}{\log L}$}  \log \IP\big[&\mbox{there exist components of $B_0 \cap S_e \cap E^{\ge \alpha}$ and $B_e \cap S_e \cap E^{\ge \alpha}$}\label{6.2}
\\[-1ex]
&\mbox{of diameter at least $\frac{L}{10}$, which are not connected in} \nonumber
\\
& D \cap S_e \cap E^{\ge \beta}\big] = -\infty. \nonumber
\end{align}

\medskip\n
Loosely speaking, compared to the previous section, in this section the percolation of the excursion-set of the Gaussian free field along $L$-boxes in a column with direction $e$ will be produced along sufficiently thick two-dimensional slabs, which are translates of $S_e$. The point of working with the restriction to $\IZ^3$ of the Gaussian free field on $\IZ^d$, is that when $d$ is large, this restriction is a small perturbation (in a suitable sense) of a collection of i.i.d. Gaussian variables, see Section 3 of \cite{RodrSzni13} and (\ref{6.15}) below. Here are the two main results of this section.

\begin{theorem}\label{theo6.1} There exist $\wt{h}_0 > 0$, $L_0 \ge 100$ and a non-increasing integer-valued function $\wt{d}_0$ on $(0,\infty)$ such that
\begin{equation}\label{6.3}
\mbox{when $\wt{h}_0 > \alpha > \beta$ and $d \ge \wt{d}_0 (\alpha - \beta)$, then $\alpha,\beta$ and $L_0$ are admissible}.
\end{equation}
\end{theorem}

Actually, we will show more than a super-polynomial decay in $L$ for (\ref{6.1}), (\ref{6.2}) in the proof of Theorem \ref{theo6.1} (namely an exponential decay in $L$), see Remark \ref{rem6.4}.

\begin{theorem}\label{theo6.2} (see (\ref{0.6}) or (\ref{2.1}) for notation)

\medskip
There exists $h_0 > 0$, $d_0 \ge 3$ such that when $d \ge d_0$, for $\alpha < h_0$
\begin{equation}\label{6.4}
\limsup\limits_N \; \mbox{\f $\dis\frac{1}{N^{d-2}}$} \; \log \IP[A_N] \le -  \mbox{\f $\dis\frac{1}{2d}$} \;(h_0 - \alpha)^2 \,{\rm cap}_{\IR^d}([-1,1]^d).
\end{equation}
\end{theorem}

\n
We will first admit Theorem \ref{theo6.1} and prove Theorem \ref{theo6.2}.

\bigskip\n
{\it Proof of Theorem \ref{theo6.2} (conditionally on Theorem \ref{theo6.1})}: The proof is similar to the proof of Theorem \ref{theo5.5}. We will highlight the main changes. We choose $L_0$, $\wt{h}_0$ from Theorem \ref{theo6.1} and set
\begin{equation}\label{6.5}
h_0 = \mbox{\f $\dis\frac{\wt{h}_0}{2}$} \,.
\end{equation}
We pick $\gamma, \delta$ (in place of (\ref{5.17})) via
\begin{equation}\label{6.6}
\delta = h_0 < \gamma = \mbox{\f $\dis\frac{\wt{h}_0 + 3h_0}{4}$} \,.
\end{equation}
In place of (\ref{5.7}), (\ref{5.8}), given $L \ge 1$, an $L$-box $B = B_z$, with $z \in \IL$ ($= L \IZ^d$), and $|e| = 1$, we say that $B$ is $\psi$-good in the direction $e$ when (with $\gamma, \delta$ as in (\ref{6.6}))
\begin{equation}\label{6.7}
B \cap (z + S_e) \cap \{\psi_B \ge \gamma\} \;\mbox{contains a component of diameter at least $\frac{L}{10}$} 
\end{equation}
and for $B' = B_{z'}$, where $z' = z + Le$ (so that $z + S_e = z' + S_e$)  
\begin{align}
&\mbox{any two components of $B \cap (z + S_e) \cap \{\psi_B \ge \gamma\}$ and $B' \cap (z' + S_e) \cap \{\psi_{B'} \ge \gamma\}$} \label{6.8}
\\
&\mbox{with diameter at least $\frac{L}{10}$ are connected in $D \cap (z + S_e) \cap \{\psi_B \ge \delta\}$}. \nonumber
\end{align}

\n
We can proceed as in the proof of Proposition \ref{prop5.2}, where we choose $\alpha = \frac{3\wt{h}_0  + h_0}{4}$, $\beta = \frac{\wt{h}_0 + h_0}{2}$ and keep $\gamma,\delta$ as in (\ref{6.6}) (so that $\wt{h}_0 > \alpha > \beta > \gamma > \delta = h_0$), and find with the help of Theorem \ref{theo6.1} that for $d \ge \wt{d}_0 (\frac{\wt{h}_0 - h_0}{16}) \stackrel{(\ref{6.5})}{=} \wt{d}_0 (\frac{\wt{h}_0}{32})$, for any unit vector $e$ in $\IZ^d$,
\begin{equation}\label{6.9}
\lim\limits_L \; \mbox{\f $\dis\frac{1}{\log L}$} \;\log \IP[B_0 \;\mbox{is $\psi$-bad in the direction $e$}] = -\infty.
\end{equation}
We thus set
\begin{equation}\label{6.10}
d_0 = \wt{d}_0 \Big(\mbox{\f $\dis\frac{\wt{h}_0}{32}$}\Big)\,,
\end{equation}

\medskip\n
so that (\ref{6.9}) holds when $d \ge d_0$, and assume from now on that $d \ge d_0$.

\medskip
With $\alpha < h_0$ from the statement of Theorem \ref{theo6.2}, we choose $a$ similarly to (\ref{5.18}) via
\begin{equation}\label{6.11}
\alpha + a = \delta \; \big( \stackrel{(\ref{6.6})}{=} h_0\big),
\end{equation}
and define the notion of $h$-good (at level $a$) as in (\ref{5.9}).

\medskip
We have similar independence properties as in (\ref{5.13}). However, when $B^i = z + i \,L e + B_0$, $0 \le i \le n$, with $z \in \IL$ and $|e| = 1$, is a sequence of $L$-boxes $\psi$-good in the direction $e$, and $h$-good, then, as in (\ref{5.14}) (note that $\delta - a = \alpha$, by (\ref{6.11}))
\begin{equation}\label{6.12}
\begin{array}{l}
\mbox{there exists a path in $E^{\ge \delta - a} \cap \big(\bigcup\limits^n_{i = 0} D^i\big) \cap (z + S_e)$ starting in}
\\
\mbox{$B^0 \cap (z + S_e)$ and ending in $B^n \cap (z + S_e)$}
\end{array}
\end{equation}
(as in (\ref{5.14}) $D^i$ stands for the $D$-type box attached to $B^i$, see (\ref{4.3}), (\ref{4.4})).

\medskip
One picks $L$ as in (\ref{5.16}) and $\eta, \rho$ as in (\ref{5.18}), adding ``in the direction $e$'' in the description of the event under the probability, when defining $\eta$ (the specific choice of $e$ does not change the probability). One modifies the definition of $C_N$ in (\ref{5.19}) by considering for each column the face $F_{e,N}$ to which it is attached, and requiring that it contains a box, which is $\psi$-bad in the direction $e$. Analogously to (\ref{5.20}), the event $\wh{C}_N$ defined in this fashion satisfies
\begin{equation}\label{6.13}
\lim\limits_N \, \mbox{\f $\dis\frac{1}{N^{d-2}}$} \; \log \IP[\wh{C}_N] = -\infty\,.
\end{equation}
The proof of Theorem \ref{theo6.2} then proceeds along the same lines as the proof of Theorem \ref{theo5.5}. \hfill $\square$

\medskip
There now remains to prove Theorem \ref{theo6.1}.

\bigskip\n
{\it Proof of Theorem \ref{theo6.1}}: We assume throughout that $d \ge 6$, and without loss of generality, by symmetry, that $\wt{e}_i = e_i$, $i = 1,2,3$ (with $e_i$, $1 \le i \le d$, the canonical basis of $\IR^d$), so that $e = e_1$ and $H_e = \IZ e_1 + \IZ e_2 + \IZ e _3$ ($\subseteq \IZ^d$).

\medskip
Following \cite{RodrSzni13}, see (3.14) and Lemma 3.2 there, we enlarge the probability space so that $(\eta_x)_{x \in H_e}$ and $(\xi_x)_{x \in H_e}$ are independent centered Gaussian fields, with $\eta_x$, $x \in H_e$, iid centered Gaussian variables with variance $\sigma^2(d)$, where $\frac{1}{2} \le \sigma^2(d) < 1$, and $\lim_{d \r \infty} \sigma^2(d) = 1$, and
\begin{equation}\label{6.14}
\IP\big[\bigcap\limits_{x \in A} \{|\xi_x| > h\}\big] \le [v(h^2d)]^{|A|}, \;\mbox{for all $A \subset \subset H_e$},
\end{equation}
with $v$: $\IR_+ \r [0,1]$ a fixed non-increasing function satisfying $\lim_{u \r \infty} v(u) = 0$, and one has
\begin{equation}\label{6.15}
\varphi_x = \eta_x + \xi_x, \; \mbox{for all $x \in H_e$}
\end{equation}

\n
(so, when $d$ is large, we can view the restriction of $\varphi$ to $H_e$ as a ``small perturbation'' of the iid random field $\eta$).

\medskip
We will now choose $\wt{h}_0$. One knows from Theorem 4.1 of \cite{CampRuss85} that $p_c^{{\rm site}}(\IZ^3)$, the critical parameter for site percolation on $\IZ^3$, is smaller than $\frac{1}{2}$, and since $\sigma^2(d) \ge \frac{1}{2}$, we can pick $\wt{h}_0 > 0$ so that for all $d \ge 6$,
\begin{equation}\label{6.16}
P\Big[X \ge 2 \mbox{\f $\dis\frac{\wt{h}_0}{\sigma(d)}$}\Big] \ge \fr \;\Big( \fr + p_c^{\rm site}(\IZ^3)\Big) \;\mbox{with $X$ a standard normal variable}.
\end{equation}

\n
This condition will enter the construction below, which involves a static renormalization procedure attached to Bernoulli percolation.

\medskip
Given $L_0 \ge 100$, we introduce the two-dimensional lattice
\begin{equation}\label{6.17}
\cL = L_0 \,\IZ e_1 + L_0\, \IZ e_2 \subseteq S_e \subseteq H_e\,.
\end{equation}
We consider the values $\alpha,\beta$ from the statement of Theorem \ref{theo6.1}, such that 
\begin{equation}\label{6.18}
\wt{h}_0 > \alpha > \beta \,.
\end{equation}

\n
We now define the random field $(\o_x)_{x \in H_e} = (1\{\eta_x \ge 2 \wt{h}_0\})_{x \in H_e}$, which solely depends on $\eta$. Given $y \in \cL$, we denote by $\cC_y(\o)$ the (possibly empty) open (i.e. value $1$) cluster in $B_{y,e} = y + [0,2L_0)e_1 + [0,2 L_0)e_2 + [0,2L_0)e_3$ containing the most vertices. If several such clusters exist, we choose one by lexicographic order. We denote by $F_y$ (an event, which solely depends on $\o$, and hence on $\eta$) the event
\begin{equation}\label{6.19}
\left\{ \begin{array}{rl}
{\rm i)} & \mbox{$\cC_y(\o)$ is a crossing cluster for $B_{y,e}$ in the first two axes directions:}
\\
& \mbox{there is an open path in $\cC_y(\o)$ with end vertices having components}
\\
&\mbox{$y \cdot e_1$ and $(y + 2 L_0 - 1)e_1$ in the $e_1$direction, and similarly an open}
\\
&\mbox{path with component $y\cdot e_2$ and $(y + 2L_0 - 1)e_2$ in the $e_2$-direction}.
\\[2ex]
{\rm ii)}& \mbox{$\cC_y(\o)$ is the only open cluster $\cC$ of $B_{y,e}$ having diameter $\ge \frac{L_0}{10}$}.
\end{array}\right.
\end{equation}

\medskip\n
We also consider $\o' = (\o'_x)_{x \in H_e}$, where $\o'_x = 1\{\eta_x \ge \frac{\alpha + \beta}{2}\}$, for $x \in H_e$, and define just as above, with $\o'$ in place of $\o$, the events $F'_y$, $y \in \cL$. Further we also define the events $G_y,y \in \cL$, in terms of $\xi$ by
\begin{equation}\label{6.20}
G_y = \Big\{|\xi_x| \le \fr \;\big((\alpha - \beta) \wedge \wt{h}_0\big), \; \mbox{for all $x \in B_{y,e}\Big\}$ (see below (\ref{6.18}) for notation)}.
\end{equation}

\n
We say that $y \in \cL$ is $\o$-good when $F_y$ holds, $\o'$-good when $F'_y$ holds, and $\xi$-good when $G_y$ holds. Note that when
\begin{equation}\label{6.21}
\begin{array}{l}
\mbox{$y_i, 0 \le i \le M$ is a path in $\cL$ (i.e. $|y_i - y_{i-1}| = L_0$, for $1 \le i \le M$) such that}
\\
\mbox{each $y_i$ is both $\o$ and $\xi$-good, then there exists a path in $E^{\ge \wt{h}_0} \subseteq E^{\ge \alpha}$}
\\
\mbox{starting in $B_{y_0,e}$ and ending in $B_{y_M,e}$ contained in $\bigcup^M_{i=0} B_{y_i,e}$ ($\subseteq S_e$)}.
\end{array}
\end{equation}
The proof is similar to Lemma 3.4 of \cite{RodrSzni13}.

\medskip
We observe that $(F_y)_{y \in \cL}$ is a $2$-dependent collection, and the same holds for $(F'_y)_{y \in \cL}$. In addition, $(\o_x)_{x \in H_e}$ is an iid collection of Bernoulli variables with success probability
\begin{equation}\label{6.22}
\IP[\o_x = 1] = \IP[\eta_x \ge 2 \wt{h}_0] \stackrel{(\ref{6.16})}{\ge} \fr \;\Big(\fr + p_c^{\rm site}(\IZ^3)\Big)> p_c^{\rm site} (\IZ^3), \;\mbox{for all $x \in H_e$}.
\end{equation}
and similarly $(\o'_x)_{x \in H_e}$ is an iid collection of Bernoulli variables with success probability
\begin{equation}\label{6.23}
\IP[\o'_x = 1] = \IP\Big[\eta_x \ge \mbox{\f $\dis\frac{\alpha + \beta}{2}$}\Big] \stackrel{(\ref{6.18})}{\ge} \IP[\o_x = 1] \ge \fr \;\Big(\fr + p_c^{\rm site}(\IZ^3)\Big), \;\mbox{for all $x \in H_e$}.
\end{equation}

\medskip
By the site-percolation version of Theorem 7.61, p.~178 of \cite{Grim99}, we know that
\begin{equation}\label{6.24}
\lim\limits_{L_0 \r \infty} \; \inf\limits_{d \ge 6} \;\IP[F_0] = 1
\end{equation}
(the proof in the above reference shows that the uniformity in the limit can be achieved thanks to the lower bound (\ref{6.22}), see p.~190-191 of \cite{Grim99}). Similarly, we have
\begin{equation}\label{6.25}
\lim\limits_{L_0 \r \infty} \; \inf\limits_{d \ge 6} \;\IP[F'_0] = 1.
\end{equation}

\n
By Theorem 7.65, p.~179 of \cite{Grim99}, or Theorem 0.0 of \cite{LiggSchoStac97}, there is a $\pi$: $[0,1] \r [0,1]$, non-decreasing, tending to $1$ at point $1$, such that if $\IP[F_y] \ge \delta$, then $(1_{F_y})_{y \in \cL}$ stochastically dominates independent Bernoulli variables with success probability $\pi(\delta)$, and the same holds for $(F'_y)_{y \in \cL}$ in place of $(F_y)_{y \in \cL}$.

\medskip
We will now choose $L_0$ and the function $\wt{d}_0(\cdot)$  that appear in the statement of Theorem \ref{theo6.1}. We pick $L_0$ as the smallest integer $\ge 100$ such that
\begin{equation}\label{6.26}
\pi\big(\inf\limits_{d \ge 6} \; \IP[F_0] \wedge \IP[F'_0]\big) \ge 1 - \mbox{\f $\dis\frac{1}{40}$} \,.
\end{equation}
We define $\wt{d}_0(\cdot)$ as the non-increasing integer-valued function on $(0,\infty)$
\begin{equation}\label{6.27}
\begin{split}
\wt{d}_0(r) = \Big[\mbox{\f $\dis\frac{4 c_{10}}{r^2 \wedge \wt{h}_0^2}$}\Big] + 1, & \; \mbox{where $c_{10} (L_0) > 0$ is such that $2(2 L_0)^3 v(u)^{\frac{1}{4}} \le \mbox{\f $\dis\frac{1}{40}$}$},
\\[-0.5ex]
&\;\mbox{for all $u \ge c_{10}$ (and $v(\cdot)$ appears below (\ref{6.14}))}.
\end{split}
\end{equation}

\medskip\n
With the above choice, we see that (recall $v$ is non-increasing)
\begin{equation}\label{6.28}
2 (2 L_0)^3 \,v(h^2d)^{\frac{1}{4}} \le \mbox{\f $\dis\frac{1}{40}$}, \; \mbox{for $d \ge \wt{d}_0 (\alpha - \beta)$ and $h \ge \fr \,\big((\alpha - \beta) \wedge \wt{h}_0\big)$}.
\end{equation}

\medskip\n
The above estimates (\ref{6.26}) and (\ref{6.28}) are crucial for the lemma below, which controls the probability of long $*$-paths in $\cL$ of bad vertices. From now on, we assume that
\begin{equation}\label{6.29}
d \ge \wt{d}_0 (\alpha - \beta).
\end{equation}
For $y \in \cL$ and $n \ge 0$, we define (see below (\ref{6.30}) for the notation)
\begin{align}
H_{y,n} = & \;\mbox{$\{y$ is connected to some $\wt{y} \in \cL$ with $|\wt{y}-y|_\infty = n L_0$, by some $*$-path}\label{6.30}
\\
& \;\mbox{in $\cL$ of vertices which are $\o$-bad or $\xi$-bad$\}$} \nonumber
\\[2ex]
H'_{y,n} \quad & \; \mbox{defined as $H_{y,n}$ with $\o$ replaced by $\o'$}. \label{6.31}
\end{align}

\begin{lemma}\label{lem6.3} (recall $d \ge \wt{d}_0 (\alpha - \beta)$)

\medskip
For all $n \ge 0$ and $y \in \cL$
\begin{align}
\IP[H_{y,n}] & \le 2^{-n}, \label{6.32}
\\[2ex]
\IP[H'_{y,n}] & \le 2^{-n}. \label{6.33}
\end{align}
\end{lemma}

\begin{proof}
The bounds follow by the same proof as Lemma 3.5 of \cite{RodrSzni13}, which is based on a Peierls-type argument.
\end{proof}

\medskip
We will use (\ref{6.32}) to check (\ref{6.1}) and (\ref{6.33}) to check (\ref{6.2}).

\medskip
We thus come back to the proof of (\ref{6.1}) and (\ref{6.2}). We begin with (\ref{6.1}). We select (deterministically) a box $\wh{B}_0 \subseteq B_0$ with center within $|\cdot|_\infty$-distance $1$ of the center of $B_0$ and side-length $[\frac{L}{5}] + 1$ $(\ge \frac{L}{5})$, see (\ref{4.3}) for the definition of $B_0$ (for parity reasons, the centers of $B_0$ and $\wh{B}_0$ may differ). We consider the sites of $\cL$ within $\wh{B}_0$, i.e. $\cL \cap \wh{B}_0$, and its union with the connected components (in $\cL$) of good (i.e. $\o$-good and $\xi$-good) neighboring sites (in $\cL$) of $\cL \cap \wh{B}_0$. When $L$ is large, this is a connected set in $\cL$ with diameter at least $\frac{L}{5} - 2L_0$. On the event that appears in (\ref{6.1}), this set does not contain any site of $\cL$ neighboring (in $\cL$) a site of $\cL \backslash B_0$ (otherwise, by (\ref{6.21}) there would be a connected component in $B_0 \cap S_e \cap E^{\ge \alpha}$ with diameter at least $\frac{L}{10}$). Hence, we can find a $*$-circuit in $\cL \cap B_0$ surrounding $\cL \cap \wh{B}_0$ consisting of bad sites. By (\ref{6.32}) this probability decays exponentially with $L$. This is more than enough to prove (\ref{6.1}).

\medskip
We now turn to the proof of (\ref{6.2}). We assume that $L \ge 10^4 L_0$. We first consider a path $(x_i)_{0 \le i \le n}$ in $B_0 \cap S_e \cap E^{\ge \alpha}$ with diameter at least $\frac{L}{100} (\ge 100 L_0)$. We consider the successive displacements of this path at $|\cdot |_\infty$-distance $\ge \frac{L_0}{10}$. As we now explain, we can choose a path $y_0,\dots,y_m$ in $\cL$ such that any $B_{y_j,e}$ (see below (\ref{6.18}) for the notation) contains a consecutive portion of the path $(x_i)_{0 \le i \le n}$ of diameter at least $\frac{L_0}{10}$, and such that $\{x_0,\dots,x_n\} \subseteq \bigcup^m_{j=0} B_{y_j,e}$.

\medskip
Indeed, we first pick $\ov{y}_0 \in \cL$ within $|\cdot|_\infty$-distance $\frac{L_0}{2}$ from $x_0$. We set $y_0 = \ov{y}_0 e_1 - L_0 e_1 - L_0 e_2$ (so $\ov{y}_0$ is in essence the ``center'' of the $\IZ e_1 + \IZ e_2$-projection of $B_{y_0,e}$, which is a $3$-dimensional box with side-length $2L_0$, see below (\ref{6.18})).  We see that $x_0 \in B_{y_0,e}$ and the $\IZ e_1 + \IZ e_2$-projection of $x_0$ is at distance at least $\frac{L_0}{2}$ from the $\IZ e_1 + \IZ e_2$-projection of $S_e \backslash B_{y_0,e}$. If $x_{n_1}$ denotes the first displacement of the path $(x_i)_{0 \le i \le n}$ at $|\cdot|_\infty$-distance at least $\frac{L_0}{10}$ from $x_0$, one can choose $\ov{y}_1,\ov{y}_2$ in $\cL$, moving from $\ov{y}_0$ one single coordinate at a time, by an amount at most $L_0$ in absolute value, so that $|\ov{y}_2 - x_{n_1}|_\infty \le \frac{L_0}{2}$ (and $\ov{y}_1$ has its second coordinate that coincides with that of $\ov{y}_0$, and its first coordinate that coincides with that of $\ov{y}_2$). One then sets $y_1 = \ov{y}_1 - L_0 e_1 - L_0 e_2$ and $y_2 = \ov{y}_2 - L_0 e_1 - L_0 e_2$. Then, each $B_{y_j,e}$, $0 \le j \le 2$ contains a consecutive portion of the path $(x_i)_{0 \le i \le n}$ of diameter at least $\frac{L_0}{10}$. One can then repeat the construction until there are no further moves at $|\cdot|_\infty$-distance $\frac{L_0}{10}$ of the path $(x_i)$ to obtain $y_0,\dots,y_m$ (one eliminates consecutive repetitions of the $y_j$ so as to obtain a path in $\cL$). This proves our claim about the construction of $y_0,\dots,y_m$.

\medskip
The observation is now that if one of the $y_j$, say $y_{j_0}$, is good~$^\prime$ (i.e. $\o'$-good and $\xi$-good, see below (\ref{6.20}) for the definition), since one of the connected components of the trace of the path $(x_i)$ in $B_{y_{j_0},e}$ has diameter $\ge \frac{L_0}{10}$ (by construction of the $y_0,\dots,y_m)$ and the $\eta$-value of $(x_i)$ in $B_{y_{j_0},e}$ is at least $\alpha - \frac{(\alpha - \beta)}{2} = \frac{\alpha + \beta}{2}$ (because $y_{j_0}$ is $\xi$-good), then the path $x$ meets the crossing cluster of the box $B_{y_{j_0},e}$ (because $y_{j_0}$ is $\o'$-good).

\medskip
As we now explain, except on a set with exponentially decaying probability in $L$,
\begin{equation}\label{6.34}
\begin{array}{l}
\mbox{any path $(x_i)_{0 \le i \le n}$ in $B_0 \cap S_e \cap E^{\ge \alpha}$ with diameter $\ge \frac{L}{100}$ meets the}
\\
\mbox{component generated by the crossing clusters of some component in $\cL$}
\\
\mbox{of good~$^\prime$  sites with diameter at least $\frac{L}{5}$}.
\end{array}
\end{equation}

\n
Indeed, on the complementary event, pick $(x_i)_{0 \le i \le n}$ in $B_0 \cap S_e \cap E^{\ge \alpha}$ with diameter at least $\frac{L}{100}$, which does not meet any component of the type that appears in (\ref{6.34}). We can attach $(y_j)_{0 \le j \le m}$ to $(x_i)_{0 \le i \le n}$. Since $\{x_0, \dots, x_n\} \subseteq \bigcup^m_{j= 0} B_{y_j,e}$, the connected subset $\{y_0,\dots,y_m\}$ of $\cL$ has diameter at least $\frac{L}{100} - 4L_0$. We enlarge this connected set in $\cL$ by adding all components (in $\cL$) of good~$^\prime$ sites of the $y_j$, $0 \le j \le m$. None of these components (in $\cL$) has diameter $\ge \frac{L}{5}$ (otherwise $(x_i)_{0 \le i \le n}$ would meet the component generated by the crossing clusters of good~$^\prime$ sites, by the observation above (\ref{6.34}), and this would contradict the fact that, by assumption, $(x_i)_{0 \le i \le n}$ does not meet the connected component generated by the crossing clusters of any component in $\cL$ of good~$^\prime$ sites with diameter at least $\frac{L}{5}$). We can then consider bad~$^\prime$ vertices that arise either as one of the $y_j,0 \le j \le m$, or in the exterior boundary (in $\cL$) of these added components, and find a $*$-connected path in $\cL$ of bad~$^\prime$ vertices, which remains in the $L$-neighborhood of $B_0$ and has diameter at least $\frac{L}{100} - 4L_0 \ge \frac{L}{200}$. By (\ref{6.33}) this event has a probability which decays exponentially in $L$. The claim (\ref{6.34}) follows.

\medskip
Thus, except on a set with exponentially decaying probability in $L$, any two paths $x_0,\dots,x_n$ in $B_0 \cap S_e \cap E^{\ge \alpha}$ and $x'_0,\dots,x'_{n'}$ in $B_e \cap S_e \cap E^{\ge \alpha}$ with respective diameter at least $\frac{L}{10}$ each meet the crossing cluster of some component (in $\cL$) of good~$^\prime$ vertices having diameter at least $\frac{L}{5}$. As a result, on the event that appears in (\ref{6.2}), except on an event with exponentially decaying probability in $L$, there exist within $\{y \in D \cap \cL$; $d(y,D^c) \ge 2 L_0\}$ two distinct connected components of good~$^\prime$ vertices with diameter at least $\frac{L}{5}$. Hence, there is a $*$-connected path in $\cL$, starting in $D \cap \cL$, with diameter $\ge \frac{L}{5}$, consisting of bad~$^\prime$ sites. By (\ref{6.33}) this event has exponentially decaying probability in $L$. This is more than enough to prove (\ref{6.2}). This concludes the proof of Theorem \ref{theo6.1}. \hfill $\square$

\begin{remark}\label{rem6.4} \rm
Although we did not need this fact for the proof of Theorem \ref{theo6.2} (which was closely following the proof of Theorem \ref{theo5.5}), the above proof shows that the statement of Theorem \ref{theo6.1} remains true if one replaces the super-polynomial decay in (\ref{6.1}) and (\ref{6.2}) by an exponential decay. \hfill $\square$
\end{remark}

\section{An application to random interlacements and to simple random walk}
\setcounter{equation}{0}

In this section we apply the results of Sections 5 and 6 to the derivation of an upperbound on the probability that random interlacements or simple random walk disconnect $\partial B_N$ from $S_N$. In particular, our controls capture an exponential decay at rate $N^{d-2}$ of the probability of such a disconnection, when the level of the interlacement is low and the dimension large enough, and in the case of simple random walk, when the dimension is large enough. As far as we are aware, this goes beyond the current state of knowledge. The bounds of this section also provide one further motivation for showing that the critical parameter $\ov{h}$ from (\ref{5.3}) is positive. Our main transfer mechanism stems from a recent version due to Lupu \cite{Lupu} of the isomorphism theorem linking random interlacements and the Gaussian free field, see \cite{Szni12b}.

\medskip
We denote by $\cI^u$ the random interlacement at level $u \ge 0$ in $\IZ^d$, $d \ge 3$, and by $\cV^u = \IZ^d \backslash \cI^u$, the vacant set at level $u$, see \cite{Szni10a}. Motivated by large deviation estimates of \cite{LiSzni15} on the profile of occupation-times of random interlacements, lower bounds for the probability of the disconnection of a macroscopic body by random interlacements, when the vacant set is in a percolative regime, have been derived in the subsequent article \cite{LiSzni14}. From Theorem 0.1 of \cite{LiSzni14} one knows for instance that for $0 \le u \le u_{**}$,
\begin{equation}\label{7.1}
\liminf\limits_N \; \mbox{\f $\dis\frac{1}{N^{d-2}}$} \; \log \IP\big[\partial B_N \overset{^{\mbox{\footnotesize $\cV^u$}}}{\mbox{\Large $\longleftrightarrow$}} \hspace{-4ex}/ \quad \;\infty\big] \ge - \mbox{\f $\dis\frac{1}{d}$} \; (\sqrt{u_{**}} - \sqrt{u})^2 {\rm cap}_{\IR^d}([-1,1]^d) 
\end{equation}
where the critical value $u_{**} \in (0,\infty)$ can be characterized as (see \cite{Szni12a}, \cite{PopoTeix13}):
\begin{equation}\label{7.2}
u_{**} = \inf\{u > 0; \liminf\limits_L \IP \big[B_L \overset{^{\mbox{\footnotesize $\cV^u$}}}{\mbox{\Large $\longleftrightarrow$}} \hspace{-4ex}/ \quad \;\partial B_{2L}\big] = 0\}
\end{equation}

\n
(this is analogous to (\ref{0.5}) and a similar fast decay of $\IP [ 0 \stackrel{\cV^u}{\longleftrightarrow} \partial B_L]$, as in (\ref{1.19}) is known to hold for $u > u_{**}$).

\medskip
The proof of Theorem 0.1 of \cite{LiSzni14} yields as well that for any $M > 1$, in the notation of (\ref{2.1}), for $0 \le u \le u_{**}$,
\begin{equation}\label{7.3}
\liminf\limits_N \; \mbox{\f $\dis\frac{1}{N^{d-2}}$} \; \log \IP\big[\partial B_N \overset{^{\mbox{\footnotesize $\cV^u$}}}{\mbox{\Large $\longleftrightarrow$}} \hspace{-4ex}/ \quad \;S_N\big] \ge - \mbox{\f $\dis\frac{1}{d}$} \; (\sqrt{u_{**}} - \sqrt{u})^2  {\rm cap}_{\IR^d}([-1,1]^d) .
\end{equation}
Our first main result in this section is an asymptotic upper bound on the probability that appears in (\ref{7.3}).

\begin{theorem}\label{theo7.1} (see (\ref{2.1}), (\ref{5.3}) and Theorem \ref{theo6.2} for notation)

\medskip
Assume $u > 0$ and $\sqrt{2u} < \ov{h}$, then for any $M > 1$,
\begin{equation}\label{7.4}
\limsup\limits_N \; \mbox{\f $\dis\frac{1}{N^{d-2}}$} \; \log \IP\big[\partial B_N \overset{^{\mbox{\footnotesize $\cV^u$}}}{\mbox{\Large $\longleftrightarrow$}} \hspace{-4ex}/ \quad \;S_N\big] \ge - \mbox{\f $\dis\frac{1}{d}$} \; \Big(\sqrt{\textstyle{\frac{\ov{h}^2}{2}}} - \sqrt{u}\Big)^2 {\rm cap}_{\IR^d}([-1,1]^d) .
\end{equation}
Moreover, in the notation of Theorem \ref{theo6.2}, when $d \ge d_0$ and $\sqrt{2u} < h_0$, then for any $M > 1$,
\begin{equation}\label{7.5}
\limsup\limits_N \; \mbox{\f $\dis\frac{1}{N^{d-2}}$} \; \log \IP\big[\partial B_N \overset{^{\mbox{\footnotesize $\cV^u$}}}{\mbox{\Large $\longleftrightarrow$}} \hspace{-4ex}/ \quad \;S_N\big] \le - \mbox{\f $\dis\frac{1}{d}$} \; \Big(\sqrt{\textstyle{\frac{h^2_0}{2}}} - \sqrt{u}\Big)^2 {\rm cap}_{\IR^d}([-1,1]^d) .
\end{equation}
\end{theorem}

\begin{proof}
By Theorem 3 of \cite{Lupu}, given $u > 0$, one can find a coupling $P$ of $\cI^u$ and the Gaussian free field $\varphi$ so that $P$-a.s.,
\begin{equation}\label{7.6}
\cV^u \supseteq E^{\ge \sqrt{2u}} \big( = \big\{x \in \IZ^d; \varphi_x \ge \sqrt{2u}\big\}\big).
\end{equation}
Thus, choosing $\alpha = \sqrt{2u}$, we find that
\begin{equation}\label{7.7}
\IP\big[\partial B_N \overset{^{\mbox{\footnotesize $\cV^u$}}}{\mbox{\Large $\longleftrightarrow$}} \hspace{-4ex}/ \quad \;S_N\big] \le P\big[\partial B_N \overset{^{\mbox{\footnotesize $E^{\ge \alpha}$}}}{\mbox{\Large $\longleftrightarrow$}} \hspace{-4ex}/ \quad \;S_N\big] = \IP[A_N].
\end{equation}

\medskip\n
The claims (\ref{7.4}) and (\ref{7.5}) are now direct consequences of Theorems \ref{theo5.5} and \ref{theo6.2}. 
\end{proof}

\begin{remark}\label{rem7.2}  \rm ~

\medskip\n
1) The upper bound (\ref{7.4}) is only meaningful if $\ov{h} > 0$, and in this respect (\ref{7.4}) yields a further motivation to prove that $\ov{h}$ is positive.

\medskip\n
2) From Theorem 3 of \cite{Lupu}, se also (\ref{7.6}) above, one knows that $h_* \le \sqrt{2u}_*$ and also that $h_{**} \le \sqrt{2u_{**}}$. These inequalities are conceivably strict inequalities. If this is indeed the case, since $\ov{h} \le h_{**}$, cf.~(\ref{5.4}), the asymptotic upper bound (\ref{7.4}) will not match the asymptotic lower bound (\ref{7.3}). So, quite possibly, the transfer mechanism (\ref{7.6}) is insufficient to obtain matching asymptotic upper and lower bounds for disconnection by random interlacements. Still, one can wonder whether $u_* = u_{**}$ actually holds and (\ref{7.3}) is complemented by a matching upper bound (see also Remark 6.5 5) of \cite{LiSzni15}), so that in fact, for $0 < u < u_*$,
\begin{equation}\label{7.8a} 
\lim\limits_N \;\mbox{\f $\dis\frac{1}{N^{d-2}}$} \;\log \IP [ \partial B_N \overset{^{\mbox{\footnotesize $\cV^u$}}}{\mbox{\Large $\longleftrightarrow$}} \hspace{-4ex}/ \quad \;S_N] = - \mbox{\f $\dis\frac{1}{d}$} \,(\sqrt{u}_* - \sqrt{u})^2 \,{\rm cap}_{\IR^d}([-1,1]^d)\,?
\end{equation}

 \hfill $\square$
\end{remark}

\medskip
We will now deduce from Theorem \ref{theo7.1} an upper bound on the probability of disconnection of $\partial B_N$ from $S_N$ by the simple random walk. We also refer to (\ref{1.10}) and Theorem 6.1 of \cite{Wind08b} for a related question. We denote by $\cI$ the set of points in $\IZ^d$ visited by the simple random walk and by $\cV = \IZ^d \backslash \cI$ its complement. We recall that $P_0$ stands for the canonical law of the walk starting from the origin (see the beginning of Section 1 for notation).

\begin{corollary}\label{cor7.3} (see (\ref{2.1}), (\ref{5.3}) and Theorem \ref{theo6.2} for notation)

\medskip
If $\ov{h} > 0$, then for any $M > 1$,
\begin{equation}\label{7.8}
\limsup\limits_N \;\mbox{\f $\dis\frac{1}{N^{d-2}}$} \log P_0 \big[\partial B_N \overset{^{\mbox{\footnotesize $\cV$}}}{\mbox{\Large $\longleftrightarrow$}} \hspace{-3.8ex}/ \quad \;S_N\big]  \le - \mbox{\f $\dis\frac{1}{2d}$} \;\ov{h}\,^{\!2} \,{\rm cap}([-1,1]^d).
\end{equation}
Moreover, when $d \ge d_0$, then for any $M > 1$,
\begin{equation}\label{7.9}
\limsup\limits_N \;\mbox{\f $\dis\frac{1}{N^{d-2}}$} \log P_0 \big[\partial B_N \overset{^{\mbox{\footnotesize $\cV$}}}{\mbox{\Large $\longleftrightarrow$}} \hspace{-3.8ex}/ \quad \;S_N\big]  \le - \mbox{\f $\dis\frac{1}{2d}$} \;h_0^2\, {\rm cap}([-1,1]^d).
\end{equation}
\end{corollary}

\begin{proof}
Consider $u > 0$, then $\cI^u$, the random interlacement at level $u$, is the trace on $\IZ^d$ of a certain Poisson cloud of bilateral trajectories modulo time-shift, see \cite{Szni10a}, Section 1. The forward trajectories induced by the bilateral trajectories modulo time-shift passing through the origin (in $\IZ^d$), onward from their first visit at zero, is a Poisson point process with intensity measure $\frac{u}{g(0)} P_0$. In particular, the number of trajectories going through $0$ is a Poisson variable with parameter $\frac{u}{g(0)}$. We can thus find a coupling $P$ of the simple random walk $X$ starting from the origin (i.e.~under $P_0$), and the random interlacement $\cI^u$ conditioned on containing the origin (i.e.~under $\IP[\cdot |0 \in \cI^u]$) so that
\begin{equation}\label{7.10}
\mbox{$\IP$-a.s. $\cI \subseteq \cI^u$ \quad (and $\cV \supseteq \cV^u$)}.
\end{equation}
As a result we see that
\begin{equation}\label{7.11}
\begin{split}
P_0 \big[\partial B_N \overset{^{\mbox{\footnotesize $\cV$}}}{\mbox{\Large $\longleftrightarrow$}} \hspace{-3.8ex}/ \quad \;S_N\big]  & = P \big[\partial B_N \overset{^{\mbox{\footnotesize $\cV$}}}{\mbox{\Large $\longleftrightarrow$}} \hspace{-3.8ex}/ \quad \;S_N\big]  \stackrel{(\ref{7.10})}{\le} P \big[\partial B_N \overset{^{\mbox{\footnotesize $\cV^u$}}}{\mbox{\Large $\longleftrightarrow$}} \hspace{-3.8ex}/ \quad \;S_N\big] .
\\[1ex]
& = \IP \big[\partial B_N \overset{^{\mbox{\footnotesize $\cV^u$}}}{\mbox{\Large $\longleftrightarrow$}} \hspace{-3.8ex}/ \quad \;S_N \big| 0 \in \cI^u\big]
\\[1ex]
& \le (1 - e^{-\frac{u}{g(0)}})^{-1} \,\IP \big[\partial B_N \overset{^{\mbox{\footnotesize $\cV^u$}}}{\mbox{\Large $\longleftrightarrow$}} \hspace{-3.8ex}/ \quad \;S_N\big] .
\end{split}
\end{equation}

\n
Taking logarighms and dividing by $N^{d-2}$, the result follows by letting $N$ first go to infinity, and then letting $u$ tend to zero.
\end{proof}

\begin{remark}\label{rem7.4} \rm As in Remark \ref{rem7.2} 1), the upper bound (\ref{7.8}) yields further incentive to show that $\ov{h} > 0$. Actually, one can also wonder whether the following asymptotics holds (see also Remark 5.1 2) of \cite{LiSzni14})
\begin{equation}\label{7.13}
\lim\limits_N \;\mbox{\f $\dis\frac{1}{N^{d-2}}$} \log P_0 \big[\partial B_N \overset{^{\mbox{\footnotesize $\cV$}}}{\mbox{\Large $\longleftrightarrow$}} \hspace{-3.8ex}/ \quad \;S_N\big] = - \mbox{\f $\dis\frac{1}{d}$} \;u_*\, {\rm cap}_{\IR^d} ([-1,1]^d)\;?
\end{equation}
\hfill $\square$
\end{remark}


\end{document}